\documentclass[11pt]{amsart}
\usepackage{amsmath,amssymb}
\usepackage{amsfonts}
\usepackage{amsbsy}
\usepackage{fullpage}
\usepackage{stmaryrd}
\usepackage{latexsym}
\usepackage{quiver}
\usepackage{mathrsfs}
\usepackage{accents}
\usepackage[all]{xy}
\usepackage{tikz}
\usepackage{tikz-cd}
\usepackage{longtable}
\usepackage{cancel}
\usepackage{cite}
\usepackage{url}
\usepackage{hyperref}
\hypersetup{colorlinks}
\usepackage{xcolor}
\usepackage{boxedminipage}
\usepackage{framed}

\usepackage{bm}
\usepackage{enumitem}
\usepackage{datetime2}
\usepackage{comment}

\newcommand{\fm}{\mathfrak{m}}

\newcommand{\cF}{\mathcal{F}}

\newcommand{\R}{\mathbb{R}}

\newcommand{\scA}{\mathscr{A}}
\newcommand{\scB}{\mathscr{B}}
\newcommand{\scC}{\mathscr{C}}

\newcommand{\scG}{\mathscr{G}}

\newcommand{\scM}{\mathscr{M}}

\newcommand{\scO}{\mathscr{O}}

\DeclareMathAlphabet{\mathpzc}{OT1}{pzc}{m}{it}
\newcommand{\svee}{{\scriptscriptstyle{\vee}}}

\newcommand{\cin}{C^\infty}

\newcommand{\inv}{^{-1}}

\newcommand{\op}[1]{{#1}^{\mbox{\sf{\tiny{op}}}}}

\newcommand{\uu}[1]{\underline{#1}}

\newcommand{\lb}{\left[ \cdot\,,\cdot\right]}
\newcommand{\Pb}{\left\{ \cdot\,,\cdot\right\}}

\DeclareMathOperator{\Der}{Der}
\DeclareMathOperator{\cDer}{C^\infty Der}

\DeclareMathOperator{\CDer}{\cin Der}

\DeclareMathOperator{\im}{im}

\DeclareMathOperator{\Spec}{Spec}
\DeclareMathOperator{\ev}{\mathsf{ev}}
\DeclareMathOperator{\pr}{\mathsf{pr}}

\newcommand{\Euc}{\mathsf{Euc}}

\newcommand{\cring}{\cin\mathsf{Ring}}

\newcommand{\Mod}{\mathsf{Mod}}
\newcommand{\Man}{\mathsf{Man}}

\newcommand{\LCRS}{\mathsf{L\cin RS}}

\newcommand{\Set}{\mathsf{Set}}

\newcommand{\Aff}{\mathsf{\cin Aff}}

\theoremstyle{definition}
\newtheorem{thm}{Theorem}[section]
\newtheorem{lemma}[thm]{Lemma}
\newtheorem{theorem}[thm]{Theorem}
\newtheorem{proposition}[thm]{Proposition}
\newtheorem{corollary}[thm]{Corollary}
\newtheorem*{corollary*}{Corollary}

\newtheorem*{claim*}{Claim}
\newtheorem{definition}[thm]{Definition}
\newtheorem{remark}[thm]{Remark}
\newtheorem*{remark*}{Remark}
\newtheorem{example}[thm]{Example}
\newtheorem{notation}[thm]{Notation}

\numberwithin{equation}{thm}

\makeatletter
\let\@wraptoccontribs\wraptoccontribs
\makeatother

\begin{document}

\title{The symplectic leaves of the affine Poisson $C^\infty$-scheme associated with a singular subalgebroid} 
\author{Ruben Louis}
\contrib[ \lowercase{with appendix} A \lowercase{by}]{Eugene Lerman}
\address{Department of Mathematics, University of Illinois at Urbana-Champaign, 214 Harker Hall,
1305 W. Green Street, Urbana, IL 61801 USA}
\email{rlouis@illinois.edu}
\thanks{typeset \DTMnow}
\maketitle

\begin{abstract}
  Every singular subalgebroid, and in particular every singular foliation $\mathscr{G}$, determines a universal Poisson $C^\infty$-scheme $\Spec (\mathcal{F}(\mathscr{G}))$, which is the analogue of the linear Poisson structure on the dual of a Lie algebroid. We show that the underlying space $X_{\mathcal{F}(\mathscr{G})}$ of $\Spec (\mathcal{F}(\mathscr{G}))$
decomposes as a disjoint union of symplectic manifolds. Furthermore, we extend the Helffer-Nourrigat cone to a singular subalgebroid and show that it is an embedded Poisson $C^\infty$-subscheme of $\Spec (\mathcal{F}(\mathscr{G}))$.
\end{abstract}
\tableofcontents
\section{Introduction}
Lie algebroids generalize tangent bundles of manifolds and arise naturally in various areas of differential and algebraic geometry. In differential geometry, they serve as the infinitesimal counterparts of Lie groupoids, while in Poisson geometry, the cotangent bundle of a Poisson manifold carries a canonical Lie algebroid structure. Lie algebroids encompass an important subclass of examples provided by regular foliations, that is, involutive subbundles of the tangent bundle of a manifold. Specifically, foliations correspond to Lie algebroids whose anchor maps are injective. 

Lie algebroids also admit singular generalizations. One important class consists of singular subalgebroids (Definition~\ref{def:Sing-subal}), introduced by Zambon~\cite{Zambon}, which encompass singular foliations as a subclass~\cite{AS, LLL1}. More generally, the algebraic analogues of singular Lie algebroids are given by Lie-Rinehart algebras. 

In this paper, we regard $C^\infty(M)$, the algebra of smooth functions on a 
smooth manifold $M$, not merely as an $\mathbb{R}$-algebra, but as a 
$C^\infty$-ring (\S \ref{sec:Cinfty-rings}). This point of view provides a natural algebraic framework for 
understanding singular Lie algebroids while, at the same time, recovering the 
classical regular theory as a particular case.

In \cite{Courant}, T.~Courant established a one-to-one correspondence between Lie algebroid structures on a vector bundle~$A$ and linear Poisson structures on its dual bundle~$A^*$. Recently, in \cite{LL}, we extended Courant's correspondence to arbitrary Lie-Rinehart algebras over $C^\infty$-rings (Theorem \ref{thm:1-1}). This extension allows us to associate a universal
Poisson $C^\infty$-ring, denoted by $\mathcal{F}(\mathscr{G})$, with any
singular subalgebroid~$\mathscr{G}$. More precisely, let
\[
C^\infty(M)/C^\infty\text{-}\mathsf{Ring}
\]
denote the undercategory of $C^\infty$-rings associated with $C^\infty(M)$. The
forgetful functor
\[
C^\infty(M)/C^\infty\text{-}\mathsf{Ring}
\longrightarrow
C^\infty(M)\text{-}\mathsf{Mod}
\]
admits a left adjoint
\[
\mathcal{F}\colon
C^\infty(M)\text{-}\mathsf{Mod}
\longrightarrow
C^\infty(M)/C^\infty\text{-}\mathsf{Ring}.
\]
Moreover, for any vector bundle $A\to M$, applying $\mathcal{F}$ to the
$C^\infty(M)$-module of smooth sections $\Gamma(A)$ yields
\[
\mathcal{F}(\Gamma(A))= C^\infty(A^*),
\]
where $C^\infty(A^*)$ is the algebra of smooth functions on the dual
bundle~$A^*$. If $\mathscr{G}$ is a singular subalgebroid, its Lie-Rinehart algebra structure
induces a Poisson $C^\infty$-ring structure on $\mathcal{F}(\mathscr{G})$.
Consequently, its $C^\infty$-spectrum
\(
\operatorname{Spec}\bigl(\mathcal{F}(\mathscr{G})\bigr)
\)
is naturally an affine Poisson $C^\infty$-scheme. In this way, the classical
dual Poisson manifold $A^*$ associated with a Lie algebroid~$A$ admits a
natural extension to the setting of singular subalgebroids.

A natural question from the perspective of Poisson geometry is the following:
\emph{What becomes of the symplectic leaves in the singular setting? Do they
continue to exist?} The purpose of this paper is to address this question. We
show that the answer is affirmative when the singular subalgebroid
$\mathscr{G}$ is a coherent sheaf, so that $\mathscr{G}$ can be realized as
the image of a vector bundle morphism~$\rho$ with $\ker\rho$ being a locally finitely generated $C^\infty(M)$-module. In this case, the associated
Poisson $C^\infty$-ring takes the form
\[
\mathcal{F}(\mathscr{G})
=
C^\infty(A^*)/
\bigl\langle \operatorname{ev}_{\ker\rho}\bigr\rangle.
\]
where
\(
\operatorname{ev}\colon \Gamma(A)\hookrightarrow C^\infty(A^*)
\)
denotes the evaluation map. Consequently, there is a closed embedding of affine $C^\infty$-schemes
\[
\operatorname{Spec}\bigl(\mathcal{F}(\mathscr{G})\bigr)
\hookrightarrow
\operatorname{Spec}\bigl(C^\infty(A^*)\bigr)
=
(A^*,C^\infty_{A^*}).
\]
Within this realization, we show that an appropriate analogue of the
symplectic foliation persists in the singular setting. More precisely, we show the following theorem holds.

\begin{theorem}
Let $(E\to M, \lb_E, r)$ be a Lie algebroid. 
Let $\scG\subset \Gamma(E)$ be a singular subalgebroid (Definition \ref{def:Sing-subal}) so that there is an $E$-valued almost Lie algebroid $\rho\colon A\to E$ with $\rho(\Gamma(A))=\scG$. If $\mathscr{G}$ is a coherent sheaf, then 
\begin{enumerate} 
    \item \label{(4.12.i)} the underlying space $X_{\mathcal{F}(\mathscr{G})}$ of the corresponding Poisson $C^\infty$-scheme  \[\Spec( \cF(\scG)) \hookrightarrow \Spec(C^\infty(A^*) )=(A^*, C^\infty_{A^*})\]
\begin{enumerate}
    \item is a disjoint union of symplectic smooth manifolds (referred to as \emph{symplectic leaves of\; $X_{\cF(\mathcal G)}$}). 
    
    \item The symplectic leaves of $X_{\mathcal{F}(\mathcal G)}$ are independent of the choice of the almost Lie algebroid $A$. More precisely, for any other almost Lie algebroid $A'$, there is a canonical leaf-preserving Poisson $C^\infty$-scheme isomorphism
\[
X^{A}_{\mathcal{F}(\mathscr{G})}
\stackrel{\sim}{\longrightarrow}
X^{A'}_{\mathcal{F}(\mathscr{G})}.
\]
Here, $X^{A}_{\mathcal{F}(\mathscr{G})}$ and
$X^{A'}_{\mathcal{F}(\mathscr{G})}$ denote the corresponding realizations of
$X_{\mathcal{F}(\mathscr{G})}$ in $A^*$ and $A'^*$, respectively.
\end{enumerate}
\item \label{(4.12.ii)} for each leaf $L\subset M$ of the singular foliation $r(\mathscr{G})\subset \mathfrak X(M)$, the cotangent bundle $T^*L\hookrightarrow A^*|_L$ is a symplectic leaf of $X_{\cF(\scG)}$. 
\item for $M_{\mathrm{reg}}^\mathscr{G}\subset M$ the set of regular points of $\mathscr{G}$, \[X_{\mathcal{F}(\mathscr{G})}\cap A^*|_{{M}_{\mathrm{reg}}^{\mathscr{G}}}=\bigcup_{m\in M_{\mathrm{reg}}^{\mathscr{G}}}\mathrm{Im}(\rho^*_m)\supseteq \bigcup_{\substack{ m\in M_{\mathrm{reg}}^{\mathscr{G}}\cap L\\
L\in M/\mathscr{G}}}T^*_mL\] with equality when $E=TM$ and $r=\mathrm{id}$.
\end{enumerate}
\end{theorem}
Besides the Poisson $C^\infty$-ring $\mathcal{F}(\mathscr{G})$ associated 
with a singular foliation $\mathscr{G}\subset \mathfrak{X}(M)$, there is 
another canonical Poisson $C^\infty$-scheme naturally associated with 
$\mathscr{G}$, namely the Helffer-Nourrigat cone 
\cite{HN,AMY,Lo3}. The latter arises in the study of longitudinal differential operators on singular foliations. We compare these two Poisson $C^\infty$-schemes. 
We show in Theorem~\ref{thm:HNcone} that the Helffer-Nourrigat cone 
embeds into $\Spec(\mathcal{F}(\mathscr{G}))$ as a Poisson 
$C^\infty$-subscheme. In particular, its underlying space
\[
    \mathrm{HN}(\mathscr{G}) \subset X_{\mathcal{F}(\mathscr{G})}
\]
is a disjoint union of symplectic leaves of 
$X_{\mathcal{F}(\mathscr{G})}$. This provides a natural interpretation of 
the Helffer-Nourrigat cone within the framework of Poisson 
$C^\infty$-rings and $C^\infty$-schemes.

The spaces $\mathrm{HN}(\mathscr{G})$ and 
$X_{\mathcal{F}(\mathscr{G})}$ are, however, rather different in 
nature. The former is a cone, as its name suggests: its fibers are 
invariant under scalar multiplication. The latter is instead naturally 
viewed as a singular vector bundle, in the sense that its fibers are 
vector spaces whose dimensions may vary from point to point (Remark \ref{Rmk:Tepui}). 
Consequently, the inclusion
\[
    \mathrm{HN}(\mathscr{G}) \subset X_{\mathcal{F}(\mathscr{G})}
\]
is, in general, strict, as illustrated in 
Example~\ref{ex:comparison}.

The paper is organized as follows. In \S~\ref{sec:2}, we review the 
definitions of singular subalgebroids, singular foliations, and almost 
Lie algebroid structures. In \S~\ref{sec:3} and \S~\ref{sec:4}, we 
review the necessary background on Poisson $C^\infty$-rings and 
$C^\infty$-schemes, and recall the one-to-one correspondence between 
Lie-Rinehart algebras over $C^\infty$-rings and linear Poisson 
$C^\infty$-rings. Finally, in \S~\ref{sec:5} and \S~\ref{sec:6}, 
we state and prove the main results of the paper.

\vspace{0.2cm}

{\bf Acknowledgments.} I would like to thank Stefano Ronchi, Chenchang Zhu, and Camilo Angulo for their stimulating questions and fruitful discussions in the early stages of this project, particularly during my visits to Göttingen in the summer of 2026. I am also grateful to Marco Zambon and Camille Laurent-Gengoux for their valuable feedback on the results, and to Rui Fernandes for helpful discussions. Finally, I would like to thank Eugene Lerman for writing Appendix \ref{appen:technical-lem}.

\vspace{0.2cm}

{\bf AI Declaration.} Large language models (LLMs) were used to improve the english and identify typographical errors in the manuscript. They were also used to assist in identifying relevant references in the mathematical literature.

\section{Singular subalgebroids and singular foliations}\label{sec:2} Throughout this paper, we work in the smooth setting and identify a sheaf with its global sections.

The following concept was first introduced in \cite{Zambon}. Let $(E, \left[\cdot\,,\cdot\right]_E, r)$ be a Lie algebroid on a manifold $M$.

\begin{definition}\label{def:Sing-subal}
    A \emph{singular subalgebroid} of $E$  is a $C^\infty(M)$-submodule $\scG\subset\Gamma(E)$ of sections of $E$ that fulfills the following conditions:
\begin{enumerate}
     \item \textbf{Stability under Lie bracket}: $[\scG,\scG]_E\subset \scG$.
     \item \textbf{Local finite generateness}: every $m\in M$ admits an
open neighborhood $\mathcal{U}\subseteq M$ together with a finite number of sections $X_1, \ldots, X_k \in \scG$ such that for
every open subset $\mathcal V \subseteq \mathcal U$ the sections ${X_1}_{|_\mathcal{V}} ,\ldots, {X_k}_{|_\mathcal{V}}$ generate $\scG$ on $\mathcal{V}$ as a $C^\infty(\mathcal V)$-module.
 \end{enumerate}
 \end{definition} 
\begin{example}[Singular foliations] An important class of singular subalgebroids that we consider in this section is that of singular foliations in the sense of \cite{AS, LLL1}. A \emph{singular foliation} $\scG$ on $M$ is a singular subalgebroid of $E=TM$.
\end{example}


\noindent
\textbf{Three notions of regular points.}
Given a singular subalgebroid $\mathscr{G}\subseteq \Gamma(E)$, we consider the following two singular foliations:
\begin{enumerate}
    \item The singular foliation associated with $E$, namely,
    $r(\Gamma(E))\subseteq \mathfrak{X}(M)$.
\item The singular foliation associated with $\mathscr{G}$, namely,
    $r(\mathscr{G})\subseteq r(\Gamma(E)) \subseteq \mathfrak{X}(M)$. We call $r(\mathscr{G})$ the \emph{basic} singular foliation induced by $\mathscr{G}$.
\end{enumerate}
These give rise to three notions of regular points:
\begin{enumerate}
    \item The regular points of the Lie algebroid $E$, i.e., those of the singular foliation $r(\Gamma(E))\subset \mathfrak X(M)$.

    \item The regular points of the singular foliation $r(\mathscr{G})$.

    \item The open dense subset of $M$ consisting of points that admit an open neighborhood $\mathcal{U}\subseteq M$ such that the evaluation of $\mathscr{G}$ over $\mathcal{U}$ is a vector subbundle of $E|_{\mathcal{U}}$. We denote this subset by $M_{\mathrm{reg}}^{\mathscr{G}}$ and refer to its points as the \emph{regular points of $\mathscr{G}$}.
\end{enumerate}
When $E=TM$ and $r=\mathrm{id}$, the singular subalgebroid $\mathscr{G}$ is simply a singular foliation, and $M_{\mathrm{reg}}^{\mathscr{G}}$ coincides with the usual regular part of a singular foliation, in the sense of, e.g.,  \cite{AS,LLL1}. In this case, we simply write $M_{\mathrm{reg}}$ if there no confusion.

\begin{definition}Let $\scG\subseteq \Gamma(E)$ be a singular subalgebroid. An \emph{$E$-valued almost Lie algebroid  $(A, \left[\cdot\,,\cdot\right]_A,\rho)$ over $\scG$} is a vector bundle morphism  $A\stackrel{\rho}{\longrightarrow} E$ (referred to as the \emph{anchor map}) such that $\rho(\Gamma(A))=\scG$; together with a skew-symmetric bilinear map  $\left[\cdot\,,\cdot\right]_A\colon \Gamma(A)\wedge\Gamma(A)\rightarrow \Gamma(A)$ (referred to as \emph{the bracket}) satisfying for all $a,b\in \Gamma(A)$ and $f\in C^\infty(M)$
\begin{enumerate}

        \item $\rho(\left[a,b\right]_A)=[\rho(a),\rho(b)]_E$;
        \item $\left[a,fb\right]_A=r\circ\rho(a)[f]b+ f \left[a,b\right]_A$.
    \end{enumerate}
$\scG$ shall be called the \emph{basic} singular subalgebroid of $(A, \left[\cdot\,,\cdot\right]_A,\rho)$. If $E=TM$ we simply say “almost Lie algebroid”.
\end{definition}
If $\left[\cdot\,,\cdot\right]_A$ satisfies the Jacobi identity, i.e., for  $a,b, c\in \Gamma(A)$\begin{equation}
    [[a,b]_A,c]_A+ [[c,a]_A,b]_A+ [[b,c]_A,a]_A=0,
\end{equation}then  $(A, \lb_A, \rho')$ is a {Lie algebroid} over $M$ with anchor map $\rho':=r\circ \rho$.

\begin{remark}
    Every finitely generated singular subalgebroid admits an $E$-valued almost Lie algebroid \cite{LLS}.
\end{remark}
\begin{proposition}\label{prop:equivalence}
For any two  almost algebroids $(A_1, \left[\cdot\,,\cdot\right]_{A_1},\rho_{A_1})$ and $(A_2, \left[\cdot\,,\cdot\right]_{A_2},\rho_{A_2})$ over a singular subalgebroid $\mathscr{G}\subseteq \Gamma(E)$,  there exist morphisms of vector bundles covering the identity map \begin{equation}
    \xymatrix{A_1\ar@<2pt>[r]^\Phi&A_2\ar@<2pt>[l]^\Psi}
\end{equation}
such that
\begin{enumerate}
    \item $\rho_{A_2}\circ \Phi=\rho_{A_1}$ and $\rho_{A_1}\circ \Psi=\rho_{A_2}$.
      \item $\rho_{A_2}(\Phi\circ\Psi -\mathrm{id})=0$ and $\rho_{A_1}(\Psi\circ\Phi -\mathrm{id})=0$.
\end{enumerate}
In particular, $[ \Phi(a), \Phi(b)]_{A_2}- \Phi([a,b]_{A_1} ) \in \ker \rho_{A_2}$ for all $a,b \in \Gamma(A_1)$. Also, we have $[ \Psi(a), \Psi(b)]_{A_1}- \Psi([a,b]_{A_2} ) \in \ker \rho_{A_1}$ for all $a,b \in \Gamma(A_2)$.
\end{proposition}
\begin{proof}
For any pair of $E$-valued almost Lie algebroids $ (A_1, [\cdot\,,\cdot]_{A_1}, \rho_{A_1})$ and $({A_2},[\cdot\,,\cdot]_{A_2},\rho_{A_2})$ over $\mathscr{G}$, there exists (by projectivity of the module of sections of vector bundles) vector bundle morphisms \begin{equation*}
    \xymatrix{(A_1, [\cdot\,,\cdot]_{A_1}, \rho_{A_1})\ar@<2pt>[r]^\Phi&(A_2, [\cdot\,,\cdot]_{A_2}, \rho_{A_2})\ar@<2pt>[l]^\Psi}
\end{equation*}
that satisfies item 1. The other properties are automatically verified.
\end{proof}

We will need the dual version of Proposition \ref{prop:equivalence}.
\begin{corollary}\label{cor:equivalence}
    Let $(A_1, [\cdot\,,\cdot]_{A_1}, \rho_{A_1})$ and
$(A_2, [\cdot\,,\cdot]_{A_2}, \rho_{A_2})$ be almost algebroids over a
singular subalgebroid $\mathscr{G}\subseteq \Gamma(E)$. Let $\Phi$ and $\Psi$
be as in Proposition~\ref{prop:equivalence}. Then the dual maps
\[
    \xymatrix{
    A_2^* \ar@<2pt>[r]^{\Phi^\svee} &
    A_1^* \ar@<2pt>[l]^{\Psi^\svee}
    }
\]
induce the pullback maps
\[
    \xymatrix{
    C^\infty(A_1^*) \ar@<2pt>[r]^{(\Phi^\svee)^*} &
    C^\infty(A_2^*) \ar@<2pt>[l]^{(\Psi^\svee)^*}
    },
\]
which satisfy
\[
(\Phi^\svee)^*(\mathrm{ev}_a)
= \mathrm{ev}_{\Phi(a)}\quad\text{and}\quad (\Phi^\svee)^*(\mathrm{pr}^*_2f)=\mathrm{pr}^*_1f.
\]
In particular,
\begin{align}
    &(\Phi^\svee)^*(\mathrm{ev}_{\ker\rho_{A_1}})
    \subset \mathrm{ev}_{\ker\rho_{A_2}}\\& (\Phi^\svee)^*(\mathrm{ev}_{[a,b]_{A_1}})
    - \mathrm{ev}_{[\Phi(a),\Phi(b)]_{A_2}}
    \in \mathrm{ev}_{\ker\rho_{A_2}};\quad (\Phi^\svee)^*(\mathrm{pr}_{1}^*(\rho_{A_1}(a)f))
    = \mathrm{pr}^*_{2}(\rho_{A_2}(\Phi(a))f)
\end{align}
{for} $a, b\in \Gamma(A_1),\, f\in C^\infty(M)$. Also,  $(\Psi^\svee)^*\circ(\Phi^\svee)^*-\mathrm{id}$ is $(\mathrm{ev}_{\ker\rho_{A_1}})$-valued. Analogous relations hold for $(\Psi^\svee)^*$. Here, for $i=1, 2$, $\mathrm{pr}_i\colon A^*_i\to M$ is the projection map, and $\mathrm{ev}\colon \Gamma(A_i)\hookrightarrow C^\infty(A_i^*)$ is the evaluation map: \(\mathrm{ev}_a\colon A^*_i\to \mathbb R,\; \alpha\mapsto \langle \alpha, a\circ \mathrm{pr}_i(\alpha)\rangle\). 
\end{corollary}

\begin{definition}[Strong-kernel]
    Let $T\colon A\rightarrow E$ be a vector bundle morphism covering the identity map. The \emph{strong-kernel} $\mathrm{Sker}\,T_m$ of $T$ at $m\in M$ is the subvector space of $\ker T_m$ defined as $u\in A_m$ such that there is an open neighborhood $\mathcal{U}$ of $m$ and a local section $\Tilde{u}\in\Gamma(A)|_{\mathcal{U}}$ through $u$ so that $T(\Tilde{u})=0$.
\end{definition}
Here are some easy properties of the strong-kernel
\begin{proposition}\cite{LLL2,AS}
 Let $T\colon A\rightarrow E$ be a vector bundle morphism covering the identity map.

 \begin{enumerate}
     \item The map $m\mapsto \dim (\mathrm{Sker}\,T_m)$ is lower semi-continuous and $m\mapsto \dim (\ker T_m)$ is upper semi-continuous.

     \item  $\mathrm{Sker}\,T_m=\ker T_m$ if and only if there is an open neighborhood $\mathcal{U}$ of $m$ such that $\mathrm{Sker}\,T_{m'}=\ker T_{m'}$ for all $m'\in \mathcal{U}$.
     \item The open dense subset where $m\mapsto \dim (\ker T_m)$ and $m\mapsto \dim(\mathrm{Im}\,T_m)$ are
locally constant is the set of points $m\in M$ such that $\mathrm{Sker}\, T_m=\ker T_m$.
 \end{enumerate}
\end{proposition}

\begin{remark}Let $A\to E$ be an anchored bundle over a singular subalgebroid $\mathscr{G}$. Note that $\operatorname{Sker} \rho$ may have constant rank even when $\ker \rho$ does not. For instance, if $\mathscr{G} \stackrel{\rho}{\simeq} \Gamma(A)$ for some vector bundle $A$, then $\rho$ is injective on an open, dense subset $M_{\mathrm{reg}}$ of $M$. In this case, $\operatorname{Sker} \rho = \{0\}$, whereas $\ker \rho$ may be non-zero at points in $M \setminus M_{\mathrm{reg}}$. As an example, consider the singular foliation on $\mathbb{R}$ generated by the vector field $X$, defined by
\[
X(x) = 
\begin{cases} 
e^{-1/x^2} \partial_x & \text{if } x \neq 0, \\ 
0 & \text{if } x = 0. 
\end{cases}
\]Take $A=\mathbb{R}\times \mathbb{R}$ and $\rho(x, 1)=X(x)$. Here,  $\operatorname{Sker} \rho = \{0\}$, but $\ker\rho_0=\mathbb{R}$.
\end{remark}

\begin{definition}
The  \emph{leaves} of a singular subalgebroid $\mathscr{G}\subset \Gamma(E)$ are the leaves of its basic singular foliation $r(\mathscr{G})\subset \mathfrak X(M)$. The leaf space of $\mathscr{G}$ is denoted as $M/\mathscr{G}$.
\end{definition}

\section{Poisson $C^\infty$-schemes}\label{sec:3}
In this section, we review Poisson $C^\infty$-rings and $C^\infty$-schemes to provide necessary context for the main results of \S\ref{sec:5} and \S\ref{sec:6}. Further details and technical results are presented in Appendix~\ref{appen:technical-lem}.
\subsection{$C^\infty$-rings and schemes}\label{sec:Cinfty-rings}

\begin{definition}
A \emph{$C^\infty$-ring} is a product-preserving functor 
\[
\mathcal{A} : \Euc \to \Set
\]
where $\Euc$ is the category whose objects are the spaces $\mathbb{R}^n$, $n \ge 0$ and morphisms are smooth maps $\mathrm{Hom}_{\Euc}(\mathbb{R}^n, \mathbb{R}^m) = C^\infty(\mathbb{R}^n, \mathbb{R}^m)$.    
\end{definition}

\begin{remark}
    Equivalently a $\cin$-ring is a set $\scA$ together with $n$-ary operations ($n\geq 0$), one operation $f_\scA : \scA^n\to \scA$ for every smooth function $f\in \cin(\R^n)$.  We require that these operations are associative in an appropriate sense and that for every coordinate function $x_i\in \cin(\R^n)$ the corresponding operation $(x_i)_\scA: \scA^n\to \scA$ is the projection $\pr_i$ on the $i^{th}$ factor: $(x_i)_\scA = \pr_i$.

    Given a product preserving functor $\mathcal{A} : \Euc \to \Set$ we set $\scA = \mathcal{A}(\R^1)$. Since the functor $\mathcal{A}$ preserves products 
\[
    \scA^n = \mathcal{A}(\R^n).
\]
Given $f\in \cin(\R^n)$, i.e. a morphism $f:\R^n\to \R$ in $\Euc$, we get $f_\scA: = \mathcal{A}(f): \mathcal{A}(\R^n) = \scA^n\to \mathcal{A}(\R) = \scA$.  The associativity of operations comes from the fact that $\mathcal{A}$ is a functor.  The fact that $(x_i)_\scA = \pr_i$ follows from the fact that $\mathcal{A}$ preserves products.
\end{remark}

\begin{definition}[The category $\cring$]
    Product preserving functors from $\Euc$ to $\Set$ form a category $\cring$. The objects are product preserving functors and morphisms are natural transformations.
\end{definition}

\begin{remark}
    Equivalently the objects of $\cring$ are sets with $n$-ary operations and morphisms are maps of sets that preserve the operations. 
\end{remark}

\begin{example}
Let $M$ be a manifold.  The  set $C^\infty(M)$ of smooth functions on the manifold $M$ is a $\cin$-ring:   for a function $f \in C^\infty(\mathbb{R}^n)$ the corresponding operation is 
\[
f_{C^\infty(M)} : (C^\infty(M))^n \to C^\infty(M), \quad (a_1, \dots, a_n) \mapsto f \circ (a_1, \dots, a_n).
\]
\end{example}

\begin{remark}
A $C^\infty$-ring $\scA$ has an underlying $\mathbb{R}$-algebra:
\begin{itemize}
    \item Addition comes from $f(x,y) = x+y$
    \item Multiplication comes from $h(x,y) = xy$
    \item Scalars come from $\lambda : \mathbb{R}^0 \to \mathbb{R}^1,  \, * \mapsto \lambda(*)$ (nullary operations), giving $\lambda_\scA : \scA^0 = \{\star\} \to \scA$. Equivalently for every $\lambda\in \R$ we get an element $\lambda_\scA \in \scA$.
\end{itemize}

We will not notationally distinguish between a $\cin$-ring and its underlying $\R$-algebra.
\end{remark}

\begin{remark}
The $\cin$-ring $C^\infty(\mathbb{R}^n)$ is freely generated by the coordinate functions $(x_1, \dots, x_n)$. This is because for any  $f \in C^\infty(\mathbb{R}^n)$,
\[
f = f \circ (x_1, \dots, x_n) = f_{C^\infty(\mathbb{R}^n)}(x_1, \dots, x_n).
\]
\end{remark}

\begin{definition}
A \emph{module} over a $C^\infty$-ring $\scA$ is a module over its underlying $\mathbb{R}$-algebra.

In particular an \emph{ideal} in a $\cin$-ring $\scA$ is an ordinary ideal $I$ in the $\R$-algebra underlying $\scA$.
\end{definition}

\begin{example}
For a vector bundle $E$ over a manifold $M$, the space of sections $\Gamma(E)$ is a $C^\infty(M)$-module.
\end{example}

Another piece of evidence that the definition of modules and therefore of ideals is ``correct" is provided by the following lemma.

\begin{lemma}
For an ideal $I$ in a $\cin$-ring $\scA$ the quotient $\R$-algebra $\scA/I$ is naturally a $\cin$-ring.  In particular for any $n$, any function $f\in \cin(\R^n)$ and any $a_1,\ldots, a_n\in \scA$ the operation
\[
f_{\scA/I}(a_1+I,\ldots, a_n+I) := f_\scA(a_1,\ldots, a_n)+I
\]    
is well-defined.
\end{lemma}
\begin{proof}
   The result follows from Hadamard's lemma. See \cite{MR}, for example, for details.
\end{proof}

\begin{definition}
A \emph{$C^\infty$-derivation} of a $C^\infty$-ring $\scA$ with values in an $\scA$-module $\scM$ is an $\mathbb{R}$-linear map $v : \scA \to \scM$ such that for all $n$, all  $f \in C^\infty(\mathbb{R}^n)$ and $a_1, \dots, a_n \in \scA$:
\[
v\big(f_\scA (a_1, \dots, a_n)\big) = \sum_{i=1}^n \, (\partial_i f)_{\scA}(a_1, \dots, a_n) \cdot v(a_i).
\]
\end{definition}

\begin{example}
For a manifold $M$ the $\cin$-ring $\scA = C^\infty(M)$ of smooth functions and $\scM = \Omega^1_{\mathrm{dR}}(M)$, the module of de Rham 1-forms on the manifold $M$, the exterior derivative $d : C^\infty(M) \to \Omega^1_{\mathrm{dR}}(M)$ is a $C^\infty$-derivation.
\end{example}

\begin{notation}
    Given a module $\scM$ over a $\cin$-ring $\scA$, the module of  $C^\infty$-derivations of $\scA$ with values in $\scM$ is denoted by $\cDer(\scA, \scM)$.   If $\scM = \scA$ we abbreviate the $C^\infty$-derivations as $\cDer(\scA)$ and $\Der(\scA)$, respectively.
\end{notation}

\begin{remark}
For any manifold $M$ the $\R$-algebra derivations of the $\cin$-ring $\cin(M)$ of smooth functions on $M$ are exactly the same as $\cin$-ring derivations:
\[
\cDer(\cin(M)) = \Der(\cin(M))=\mathfrak X(M).
\]   
\end{remark}

\begin{definition} \label{def:A.17}
An \emph{$\R$-point} of a $\cin$-ring $\scA$ is a map of $\cin$-rings $p:\scA
\to \R$.
\end{definition}

\begin{remark} \label{rmrk:Pursell} An $\R$-point of a $\cin$-ring is also a map of $\R$-algebras. The converse is true as well: a map of
  $\R$-algebras $p:\scA\to \R$ is a map of $\cin$-rings.   See
  \cite[Proposition 3.6 (b), p.~33]{MR}.  The proof relies on the so
  called  Milnor's exercise (which appears to be a result of Pursell
  -- see \cite[Chapter 8]{Pursell}): for any (second countable Hausdorff) manifold $M$
  and any $\R$-algebra map $\varphi: \cin(M)\to \R$ there is a point
  $p\in M$ so that $\varphi(f) = f(p)$ for all $f\in \cin(M)$.

In particular for any (second countable Hausdorff) manifold $M$, the
set of $\R$-points of $\cin(M)$ is in bijection with the elements of
the set $M$.  
 \end{remark}

\begin{definition}\label{def:loc_ring}
  A $\cin$-ring $\scA$ is \emph{local} if it has a unique maximal
ideal $\fm$ and  the quotient $\scA/\fm$ is isomorphic to $\R$ (as
$\cin$-rings). Hence, local $C^\infty$-rings have exactly one $\R$-point.
\end{definition}

\begin{remark}
  Moerdijk and Reyes call our local $\cin$-rings {\em pointed} local
  $\cin$-rings.  
\end{remark}

\begin{remark} \label{rmrk:no-R-points}
  The zero $\cin$-ring has no $\R$-points.

There are also nonzero $\cin$-rings with no $\R$-points.  A simple
example is the quotient $\cin(\R^n)/\cin_c(\R^n)$ where $\cin_c(\R^n)$
is the ideal of the compactly supported functions.

Let $M$ be a maximal ideal of $\cin(\R^n)$ containing $\cin_c(\R^n)$.
Then $\cin(\R^n)/M$ is a field, hence $0$ is a unique maximal ideal in
$\cin(\R^n)/M$ but $\cin(\R^n)/M$ has no $\R$-points, hence is not
local in the sense of Definition~\ref{def:loc_ring}.
\end{remark}
  
\begin{lemma} \label{lem:2.20}
Let $\scC$, $\scB$ be two local $\cin$-rings and $\varphi:\scC\to
\scB$ a map of $\cin$ rings.  Then $\varphi$ maps the unique maximal
ideal of $\scC$ to the maximal ideal of $\scB$.
\end{lemma}

\begin{proof} Let $p_\scC:\scC\to \R$, $p_\scB:\scB\to \R$ denote the
  unique $\R$-points.   Then $p_\scB\circ \varphi: \scC\to \R$ is an
  $\R$-point.  Since $p_\scC$ is unique
\[
p_\scB\circ \varphi = p_\scC.
\]
\end{proof}

\begin{definition}
A \emph{$\cin$-ringed space} is a pair $(X, \scO_X)$ where $X$ is a
topological space and $\scO_X$ is a sheaf of $\cin$-rings on $X$.

A \emph{local} $\cin$-ringed space is a $\cin$-ringed space so that all
the stalks $\scO_{X,x}$ of the sheaf $\scO_X$ are local $\cin$-rings
in the sense of Definition~\ref{def:loc_ring}.
\end{definition}

\begin{definition}
A \emph{morphism} (or a \emph{map}) of local $\cin$-ringed spaces from $(X,\scO_X)$ to
$(Y,\scO_Y)$ is a pair $\uu{f}= (f, f_\#)$ where $f: X\to Y$ is continuous and
$f_\#: \scO_Y\to f_* \scO_X$ is a map of sheaves of $\cin$-rings on the space $Y$.
\end{definition}
\begin{remark}
  Note that the map $(f,f_\#): (X,\scO_X) \to (Y,\scO_Y)$ of local
  $\cin$-ringed spaces induces, for every point $x\in X$,  a map of
  stalks $f_x: \scO_{Y,f(x)} \to \scO_{X, x}$
  
  In algebraic geometry one defines a morphism of locally ringed
  spaces to be a pair $(f,f_\#): (X,\scO_X) \to (Y,\scO_Y)$ so that
  the induced maps on stalks $f_x$ are all local: they are required to
  take the unique maximal ideal to the unique maximal ideal.     For
  local $\cin$-ringed spaces this is unnecessary because any map of
  local $\cin$-rings automatically preserves maximal ideals (Lemma~\ref{lem:2.20}).
\end{remark}

\begin{definition}[The category $\LCRS$] \label{def:lcrs}
Local $\cin$-ringed spaces and their morphisms form a category which
we denote by $\LCRS$.
\end{definition}  

\begin{remark}
There is a fully faithful functor from the category $\Man$ of
$\cin$-manifolds to local $\cin$-ringed spaces.  It sends a manifold
$M$ to the pair $(M, \cin_M)$ where $\cin_M$ is the sheaf of smooth
functions on $M$.   It sends a smooth map $f:M\to N$ to $\uu{f} = (f,
f_\#): (M, \cin_M) \to (N, \cin_N)$ where $f_\#: \cin_N\to f_* \cin_M$
is $f^*$, the pull-back by $f$.  See \cite[Corollary~4.27]{Joy}.
\end{remark}

\subsection{The $\cin$-ring spectrum functor $\Spec$}

There is a (contravariant) functor $\Gamma$, the global sections
functor, from the category of local
$\cin$-ringed spaces $\LCRS$ %
to the (opposite) category  $\cring$ of
$\cin$-rings.  Namely, given a
$\cin$-ringed space $(M, \scO_M)$, the functor $\Gamma$ sends it to the ring
$\scO_M(M)$ of global sections.   Given a map $(f,f_\#):
(N, \scO_N)\to  (M,\scO_M) $ of local $\cin$-ringed spaces we have a map of $\cin$-rings
$f_\#:\scO_M(M) = f_*(\scO_N)(M) = \scO_N(f\inv (M) ) = \scO_N(N)$.  In other words
\[
\Gamma: \LCRS \to \op{\cring}
\] 
is given by
\[
\Gamma( (N, \scO_N)) \xrightarrow{(f,f_\#)} (M,\scO_M) ) :=
\scO_M(M)\xrightarrow{f_\#} \scO_N (N).
\]  
It is easy to see that $\Gamma$ is in fact a functor.  Thanks to a
theorem of Dubuc \cite{Dubuc}, the global section functor $\Gamma$ has
a right adjoint $\Spec: \op{\cring} \to \LCRS $ (which is unique up to
a unique isomorphism).

\begin{definition} \label{def:affine} An \emph{affine
  $\cin$-scheme} is a local $\cin$-ringed space isomorphic
to $\Spec(\scA)$ for some $\cin$-ring $\scA$.

The \emph{category of
  affine schemes}  $\Aff$ is the essential image of the functor $\Spec$.
\end{definition}

\begin{notation} \label{not:2.7}
Given an affine scheme $\Spec(\scA)$ we denote the underlying
topological space by $X_\scA$. To avoid the clutter we  denote  its
structure sheaf by $\scO_\scA$ (and not by $\scO_{X_\scA}$ as we have
done for general local $\cin$-ringed space).  Thus
\[
\Spec(\scA) = (X_\scA, \scO_\scA).
\]  
\end{notation}

\begin{notation}
The unit and the counit of the adjunction $\Gamma \vdash
\Spec$ are denoted by $\eta$ and $\varepsilon$, respectively.
\end{notation}

Unlike the spectrum functor  in algebraic geometry $\Spec: \op{\cring} \to
\Aff$ is {\sf  not} an equivalence of categories and  $\Gamma \vdash
\Spec$ is not an adjoint equivalence.   
For instance, there are
nonzero $\cin$-rings with {\em no} $\R$-points
(Definition~\ref{def:A.17} and
Remark~\ref{rmrk:no-R-points}).  The spectrum of such a ring is  the
empty set with the zero $\cin$-ring; which is the same as $\Spec(0)$.
In other words, there are $\cin$-rings $\scA$ for which the components
$\varepsilon_\scA :\scA \to \Gamma (\Spec(\scA))$ of the counit of adjunction are
not isomorphisms.

\subsection{Poisson $C^\infty$-schemes}
\begin{definition} \label{def:Poisson_ring}
A \emph{Poisson $\cin$-ring} is a $\cin$-ring $\scA$ together with a
Lie bracket $\{\cdot\,, \cdot\}: \scA\times \scA \to \scA$  so that for any $a\in \scA$ the map
\[
\{a, \cdot \}_\scA: \scA\to \scA, \quad b\mapsto \{a, b\}
\]  
is a $\cin$-ring derivation: for any $n>0$, any $f\in \cin(\R^n)$ and any $a_1,\ldots, a_n\in
\scA$
\begin{equation} %
  \left\{a, f_\scA(a_1,\ldots,a_n) \right\} = \sum_{i=1}^n (\partial_i
  f)_\scA(a_1,\ldots, a_n)\cdot \{ a, a_i\}.
 \end{equation} 
\end{definition}
\begin{remark}
A Poisson $C^\infty$-ring $(\mathscr{A}, \{\cdot,\cdot\})$ can also be regarded as a Poisson $\mathbb{R}$-algebra by viewing $\mathscr{A}$ with its underlying $\mathbb{R}$-algebra structure.
\end{remark}
\begin{example}
Let $(M, \pi\in \Gamma (\wedge^2 TM))$ be a Poisson manifold and
$\{f,g\} = \langle df\wedge dg , \pi\rangle $ the corresponding
bracket.  Then $(\cin(M), \{\cdot\,, \cdot \})$ is a $\cin$-Poisson
ring.  This is because $\{f,\,\cdot\}$ is a vector field (hence a $C^\infty$-derivation) for every $f\in \cin(M)$.
\end{example}  
\begin{definition}
A \emph{Poisson ideal} of a Poisson $C^\infty$-ring $\mathscr A$ is an ideal $\mathscr{I}$ of the underlying Poisson $\mathbb{R}$-algebra of  $\mathscr A$, i.e,  such that
\begin{enumerate}
    \item  $\mathscr{I}$ is an ideal of $\mathscr{A}$ (with its underlying $\mathbb{R}$-algebra structure)
    \item $\{f, a\} \in \mathscr I \quad \text{for all } f \in \mathscr I, \, a \in \mathscr A.$
\end{enumerate}
\end{definition}
\begin{definition} \label{def:Poisson_map}
A \emph{map} of Poisson $\cin$-rings is a map of $\cin$-rings that
preserve the Poisson brackets.
\end{definition}

\begin{proposition}\cite{L-Poisson}
Let $\mathscr I$ be a Poisson ideal of a Poisson $C^\infty$-ring $\mathscr A$. Then the quotient $C^\infty$-ring $\mathscr A/\mathscr I$ is a Poisson $C^\infty$-ring and the projection map $p \colon \mathscr A \to \mathscr A/\mathscr I$ is a map of Poisson $C^\infty$-rings.
\end{proposition}
\begin{definition}
A locally $C^\infty$-ringed space $(X, \mathscr{O}_X)$ is a \emph{Poisson $C^\infty$-ringed space} if $\mathscr{O}_X$ is a sheaf of Poisson $C^\infty$-rings, i.e, 
\begin{enumerate}
    \item For all open $U \subseteq X$, $\mathscr{O}_X(U)$ is a Poisson $C^\infty$-ring.
    \item For $V \subseteq U$, the restriction map 
    \[
    |^U_V \colon \mathscr{O}_X(U) \longrightarrow \mathscr{O}_X(V)
    \]
    preserves brackets, i.e., for all $f, g \in \mathscr{O}_X(U)$:
    \[
    \{f, g\}_{\mathscr{O}_X(U)} \big|_V = \{ f|_V, g|_V \}_{\mathscr{O}_X(V)}.
    \]
\end{enumerate}
\end{definition}
We end this section with the following result.
\begin{theorem}\cite{L-Poisson}
If $(\mathscr A,\{ \cdot\,, \cdot \})$ is a Poisson $C^\infty$-ring, then $\operatorname{Spec}(\mathscr A)$ is a Poisson locally $C^\infty$-ringed space. Furthermore,
\[
\operatorname{Spec} \colon \mathrm{Poisson}\text{ }\cring \longrightarrow {\mathrm{Poisson}}\text{ }\LCRS
\]
is a functor.
\end{theorem}


\section{Lie-Rinehart algebras and their  $C^\infty$- Poisson schemes
}\label{sec:4}
In this section, we recall a recent result from \cite{LL} concerning Lie-Rinehart algebras over $C^\infty$-rings. This construction is subsequently applied to singular subalgebroids.

\subsection{Lie-Rinehart algebras}
We begin by recalling the definition of a Lie-Rinehart algebra over a $C^\infty$-ring.
\begin{definition} \label{def:cin-LR}
Let $\scA$ be a $C^\infty$-ring. A \emph{Lie-Rinehart algebra} over $\scA$ is a tuple  \[((\scM, [\cdot\,, \cdot]_\scM),  \rho: \scM\to \CDer(\scA))\] where $\mathscr{M}$ is an $\mathscr{A}$-module equipped with the structure of a real Lie algebra $(\scM, [\cdot\,, \cdot]_\scM)$, and
\(
\rho:\mathscr{M}\to\CDer(\mathscr{A})
\)
is an $\mathscr{A}$-module homomorphism satisfying:
\begin{enumerate}      
        \item $\rho$ is a map of real Lie algebras: $\rho ([m_1, m_2]_\scM) = [\rho(m_1), \rho(m_2)]$ for all $m_1, m_2\in \scM$ and
        \item $[\cdot\,, \cdot]_\scM$ satisfies the Leibniz rule:  $
        [m_1, am_2]_\scM = (\rho(m_1) a) m_2 + a [m_1, m_2]_\scM$\\
        for all $m_1, m_2\in \scM$ and $a\in \mathscr A$.    
\end{enumerate}
\end{definition}
In general, the class of Lie-Rinehart algebras over a \(C^\infty\)-ring \(\mathscr{A}\) is strictly contained in the class of Lie-Rinehart algebras over the underlying \(\mathbb{R}\)-algebra of \(\mathscr{A}\). Indeed, an algebraic derivation of \(\mathscr{A}\) need not be a \(C^\infty\)-derivation. A counterexample can be found in \cite[Appendix C]{LL}.

Let us present the main class of examples that will be relevant in what follows.

\begin{example}
\begin{enumerate}
    \item (Singular subalgebroids/singular foliations). Notice that Definition \ref{def:Sing-subal} means that a singular subalgebroid $\scG\subseteq \Gamma(E)$ is a locally finitely generated Lie-Rinehart algebra whose Lie bracket and anchor map are given by the restrictions of the Lie bracket $[\cdot\,,\cdot]_E$ and the anchor map $r\colon \Gamma(E)\to \mathfrak X(M)$ to $\scG$, respectively.
    \item More generally, Lie-Rinehart algebras over the algebra of smooth functions $C^\infty(M)$ on a smooth manifold $M$ are Lie-Rinehart algebras (in the sense of Definition \ref{def:cin-LR}) over the $C^\infty$-ring $C^\infty(M)$.
\end{enumerate}
\end{example}
Before recalling some important results on which this paper relies, we begin with the following definition concerning \(C^\infty\)-rings.
\begin{definition} \label{def:A-alg}
Let $\scA$ be a $\cin$-ring.  An \emph{$\scA$-algebra}  is a $\cin$-ring $\scB$ together with a map $\varphi_\scB:\scA\to \scB$ of $\cin$-rings.  That is, an $\scA$-algebra is an object of the  undercategory (coslice category) $\scA/\cring$.
\end{definition}
\begin{remark}
$\scA$-algebras form a category: the coslice category $\scA/\cring$.
Recall that a
morphism in $\scA/\cring$ from $\scA\xrightarrow{\varphi_\scB} \scB$ to
$\scA\xrightarrow{\varphi_{\scB'}} \scB'$ is a map $f:\scB\to \scB'$ in $\cring$ so that
 the triangle $
\xy
(-0,10)*+{\scA  }="1";
(15,-8)*+{\scB'}="3";
 (-15,-8)*+{\scB}="2";
{\ar@{->}_{f} "2";"3"};
{\ar@{->}^{\varphi_{\scB'}} "1";"3"};
{\ar@{->}_{\varphi_\scB} "1";"2"};
\endxy
$ commutes (in $\cring$).
\end{remark}
We also have the category $\scA\textrm{-}\Mod$ of $\scA$-modules. Recall that
these are ordinary modules over the underlying  $\R$-algebra of the
$\cin$-ring $\scA$.

There is a forgetful functor 
\[
U: \scA/\cring \to \scA\textrm{-}\Mod
\]
of $\scA$: it sends an $\scA$-algebra $\scA\xrightarrow{\varphi}\scB$ to $\scB$, which is 
viewed as an Abelian group and an $\scA$-module by way of $\varphi$: $a\cdot b:=
\varphi(a)b$ for all $a\in \scA$ and all $b\in \scB$.

The following results are taken from \cite{LL}.
\begin{theorem}\label{thm:Universal-Cinfty-ring}
 Let $\scA$ be a $C^\infty$-ring and $\scM$ an $\scA$-module. Then, there exists a $C^\infty$-ring $\cF(\scM)$ and \begin{itemize}
    \item a map $\varphi_\mathscr M\colon \mathscr A\to \mathcal{F}(\mathscr M)$ of $C^\infty$-rings\item  an \(\mathscr A\)-linear map
\[
\eta_\mathscr M:\mathscr M\to U(\mathcal{F}(\mathscr M))
\]
\end{itemize}
such that every \(\mathscr A\)-linear map
\[
\psi:\mathscr M\to U(\mathscr B)
\]
to an \(\mathscr A\)-$C^\infty$-algebra \(\mathscr B\) factors uniquely through \(\mathcal{F}(\mathscr M)\):
\[\begin{tikzcd}
	{\mathscr M} & {U(\mathcal F(\mathscr M))} & {U(\mathscr B)}
	\arrow["{\eta_\mathscr M}", from=1-1, to=1-2]
	\arrow["\psi"', curve={height=10pt}, from=1-1, to=1-3]
	\arrow["{U(\widehat\psi)}", from=1-2, to=1-3]
\end{tikzcd}\]
\end{theorem}
In the case of a vector bundle, Theorem \ref{thm:Universal-Cinfty-ring} takes the following form.
\begin{proposition}
    If $\mathscr{M}=\Gamma(E)$ for some vector bundle $E\to M$, then there is a natural isomorphism
\[
\mathcal{F}(\Gamma(E))\simeq C^\infty(E^*).
\]
Under this identification, the structure maps
\[
\varphi_{\Gamma(E)}\colon C^\infty(M)\longrightarrow C^\infty(E^*),
\qquad
\eta_{\Gamma(E)}\colon \Gamma(E)\longrightarrow C^\infty(E^*)
\]
are given by the pullback along the bundle projection $E^*\to M$ and the evaluation map, respectively. More precisely, $\eta_{\Gamma(E)}$ assigns to each section $\alpha\in\Gamma(E)$ its associated fiberwise linear function
\[
\operatorname{ev}_{\alpha}\colon E^*\longrightarrow \mathbb{R},
\qquad
\xi\longmapsto \xi\bigl(\alpha(\mathrm{pr}(\xi))\bigr),
\]
where $\mathrm{pr}\colon E^*\to M$ denotes the bundle projection.
\end{proposition}

\begin{theorem}\label{thm:1-1}
Let $\scA$ be a $C^\infty$-ring. Let $((\scM, [\cdot\,,\cdot]_{\scM}), \rho:\scM \to \CDer(\scA))$ be a Lie-Rinehart algebra. Then, the $\cin$-ring $\mathcal F(\scM)$ carries a natural 
Poisson  $C^\infty$-ring structure $\left(\mathcal F(\scM),  \Pb_{\mathcal F(\scM)}\right)$.

Furthermore, the linear Poisson structure on the $\cin$-ring $\mathcal F(\scM)$ is uniquely determined by the following equations:

\begin{align}\label{eq:Rinehart-Poisson1}
    \left\{\eta_{\scM}(m_1), \eta_{\scM}(m_2)\right\}_{\cF(\scM)}
&=
\eta_{\scM}\left([m_1, m_2]_{\scM}\right)\\
\nonumber&\\\label{eq:Rinehart-Poisson2}\{\eta_\scM(m), \varphi_\scM(a)\}_{\mathcal{F}(\scM)}&=\varphi_\scM(\rho(m)a)\\\nonumber&\\
\label{eq:Rinehart-Poisson3}\{\varphi_\scM(a), \varphi_\scM(b)\}_{\mathcal{F}(\scM)}&=0.
\end{align}
\end{theorem}

\section{The $C^\infty$-Poisson scheme of a singular subalgebroid}\label{sec:5}

Singular subalgebroids are, in particular, Lie-Rinehart algebras over the underlying $C^\infty$-ring structure of $C^\infty(M)$. By Theorem~\ref{thm:1-1}, one can associate to such an object a Poisson  $C^\infty$-ring. This motivates the following definition.

\begin{definition}
The \emph{Poisson $C^\infty$-ring of singular subalgebroid} $\scG \subset \Gamma(E)$ is the \underline{unique} Poisson $C^\infty$-ring $\mathcal{F}(\scG)$ associated with $\scG$.
\end{definition}

\begin{remark}
The results of this paper remain valid, more generally, for arbitrary coherent sheaves of Lie-Rinehart algebras over a smooth
manifold $M$, without the need to introduce an ambient vector bundle $E$. Indeed, the
ambient bundle plays essentially no role in the arguments, except in the
definition of $E$-valued almost Lie algebroids, which can be readily
reformulated at the level of sections.
Nevertheless, since the main class of examples we have in mind consists of
Lie-Rinehart algebras arising from ambient vector bundles, including
singular foliations and, more generally, singular subalgebroids, we work
throughout in the latter framework.
\end{remark}

\begin{remark}
    In contrast to regular foliations, which may be viewed as involutive vector subbundles of the tangent bundle $TM$ (equivalently, Lie algebroids whose anchor is injective), the underlying $C^\infty(M)$-module $\scG\subset \mathfrak X(M)$ of a singular foliation is generally not projective; in particular, it does not arise as the module of sections of a vector bundle.
\end{remark}

If the singular subalgebroid is globally finitely generated, then the associated Poisson structure admits a particularly elegant geometric description.

\begin{proposition}\label{prop:Anchored-bundle_SF}
    Let $\scG$ be a singular subalgebroid on a smooth manifold $M$ so that there is a vector bundle morphism $\rho\colon A\to E$ with $\rho(\Gamma(A))=\scG$. Then, \begin{enumerate}
        \item  $\mathcal{F(\scG)}=C^\infty(A^*)/\langle \eta_A(\ker\rho )\rangle$;
        \item the surjective map $ \phi: C^\infty(A^*) \to \cF(\scG)$ has a geometric interpretation: \[
\Spec(\phi): \Spec( \cF(\scG)) \hookrightarrow \Spec(C^\infty(A^*) ) = (A^*, \cin_{A^*})
\]
is a closed embedding.
    \end{enumerate}
    Here, $\eta_A\colon \Gamma(A)\to C^\infty(A^*), s\mapsto \left(\alpha\mapsto \langle \alpha, s(\mathrm{pr}(\alpha))\rangle\right)$
 and $\mathrm{pr}\colon A^*\to M$ is the projection map.\end{proposition}
\begin{proof}See \cite[Example 6.11]{LL}.
\end{proof}

\begin{remark}\label{Rmk:Tepui}
The fibers $\{X_{\mathcal{F}(\mathscr{G})}\cap A_m\}_{m\in M}$ of the subset $X_{\mathcal{F}(\mathscr{G})}\subset A^*$ are vector spaces of potentially varying dimensions. We believe that this family should correspond to the dual bundle of the Tepui algebroid associated with $\mathscr{G}$, in the sense of \cite{GML}.
\end{remark}

\begin{remark}\label{rmk:Anchored-bundle_SF} Under the assumptions of Proposition \ref{prop:Anchored-bundle_SF}, 
    the associated Poisson structure on $\mathcal{F(\scG)}=C^\infty(A^*)/\langle \eta_A(\ker\rho )\rangle$ can be obtained as follows: Choose an $E$-valued almost Lie algebroid structure on $A$, i.e., a skew-symmetric bilinear (over $\mathbb R $) map $$[\cdot\,,\cdot]_A : \Gamma(A) \wedge \Gamma(A)  \longrightarrow \Gamma(A) $$
that satisfies the \emph{Leibniz identity},
\begin{equation}\label{eq:Leibniz}
   [s_1, f s_2]_A = r\circ\rho(s_1)[f]s_2 +f[s_1, s_2]_A, \quad\text{for all $s_1, s_2 \in \Gamma(A), f \in C^\infty(M )$}
\end{equation}
 and the \emph{anchor condition}:

\begin{equation}\label{eq:bracket-morphism}
   \rho([s_1, s_2]_A) = [\rho(s_1), \rho(s_2)], \quad\text{for all $s_1, s_2 \in \Gamma(A)$}.
\end{equation}
    
To every section $a\in \Gamma(A)$ is associated a fiberwise linear function on the total space of the dual vector bundle $\mathrm{pr}\colon A^*\to M$ given by the evaluation map \(\mathrm{ev}_a\equiv \eta_A(a)\colon A^*\to \mathbb R, \alpha\mapsto \langle \alpha, a\circ \mathrm{pr}(\alpha)\rangle\). There is a fiberwise linear bracket \(\{\cdot\,,\cdot\}_{A^*}\) on $A^*$ which is uniquely determined by the relation \begin{equation}\label{eq:Linear_Poisson}
    \{\mathrm{ev}_a, \mathrm{ev}_b\}_{A^*}=\mathrm{ev}_{[a,b]_A}, \; a, b\in \Gamma(A).
\end{equation} The relation \eqref{eq:Linear_Poisson} imposes that \[\{\mathrm{ev}_a, \mathrm{pr}^*f\}_{A^*}=\mathrm{pr}^*(r\circ\rho(a)[f])\;\text{and}\;
\;\{\mathrm{pr}^*f, \mathrm{pr}^*g\}_{A^*}=0,\]for all $a\in \Gamma(A)$ and $f, g\in C^\infty(M)$. See \cite[\S 13.5]{CFM}.

Since the Jacobiator takes values in $\eta_A(\ker (\rho\colon \Gamma(A)\to \Gamma(E))$, it follows that the biderivation  \(\{\cdot\,,\cdot\}_{A^*}\) induces a Poisson $C^\infty$-ring bracket on the quotient $\mathcal{F(\scG)}=C^\infty(A^*)/\langle \eta_A(\ker\rho )\rangle$ satisfying the desired relations.

For $G\in C^\infty(A^*)$, we define the vector field $\mathcal{H}^G:=\{G, \cdot\}_{A^*}\in \mathfrak X(A^*)$, it shall be called the \emph{almost hamiltonian vector field associated with $G$}.
\end{remark}

For every $a\in \Gamma(A)$, and $f\in C^\infty(M)$, consider the vector fields \begin{itemize}
    \item $\mathcal{H}^a:=\{\mathrm{ev}_a, \cdot\}_{A^*}\in \mathfrak X(A^*)$ associated with the evaluation function $\mathrm{ev}_a\colon A^*\to \mathbb R$, and
    \item $\mathcal{H}^{f}:=\{\mathrm{pr}^*f, \cdot\}_{A^*}\in \mathfrak X(A^*)$ associated with the pullback of $f$ along $\mathrm{pr}\colon A^*\to M$.
\end{itemize}

\begin{lemma}\label{lem:leaf-preserving}
Let $(A_1, [\cdot\,,\cdot]_{A_1}, \rho_{A_1})$ and
$(A_2, [\cdot\,,\cdot]_{A_2}, \rho_{A_2})$ be almost algebroids over a
singular subalgebroid $\mathscr{G}\subseteq \Gamma(E)$. Let $\Phi$ and $\Psi$ \begin{equation}
    \xymatrix{A_1\ar@<2pt>[r]^\Phi&A_2\ar@<2pt>[l]^\Psi}
\end{equation}
be as in Proposition~\ref{prop:equivalence}. Then,
\begin{enumerate}
\item \label{(5.6.i)} for $f\in C^\infty(M)$, $\{\mathrm{pr}_2^*f, \cdot\}_{A^*_2}$ is $\Phi^\svee$-related to $\{\mathrm{pr}^*_1f, \cdot\}_{A_2^*}$;
    \item \label{(5.6.ii)} for $a\in \Gamma(A_2)$ and $\xi\in X_{\mathcal{F}(\mathcal G)}\subset A^*_2$, \[T_\xi\Phi^\svee ( \mathcal{H}^{\Phi(a)}(\xi))=\mathcal{H}^a\circ\Phi^\svee(\xi);\] 
    \item \label{item:leaf-preserving} for all $b\in \Gamma(A_2)$, there exists $a\in \Gamma(A_1)$ such that $\mathcal H^b(\xi)=\mathcal{H}^{\Phi(a)}(\xi)$ for $\xi\in X_{\mathcal{F}(\mathscr{G})}\subset A^*_2$. In addition, for $\xi \in X_{\mathcal F(\mathcal G)} \subset A^*_2$, and whenever the flow at $t \in \R$ is defined with\footnote{In the proof of Theorem \ref{thm:main1} (Claim 1), we show that this is always the case.} $\Phi^{\mathcal H^b}_{t}(\xi) \in X_{\mathcal F(\mathcal G)}$, we have:\[\Phi^\svee\circ \Phi^{\mathcal{H}^b}_{t}(\xi)=\Phi^{\mathcal{H}^a}_{t}(\xi).\] 
\end{enumerate}
Analogous relations hold for $\Psi^\svee$.
\end{lemma}
\begin{proof}
Item (\ref{(5.6.i)}) and (\ref{(5.6.ii)}) follow from the relations \begin{align}
   &(\Phi^\svee)^*(\mathrm{pr}_{1}^*(\rho_{A_1}(a)f))
    = \mathrm{pr}^*_{2}(\rho_{A_2}(\Phi(a))f), 
    \quad \text{for}\; a\in \Gamma(A_1),\, f\in C^\infty(M);\\&(\Phi^\svee)^*(\mathrm{ev}_{[a,b]_{A_1}})
    - \mathrm{ev}_{[\Phi(a),\Phi(b)]_{A_2}}
    \in \mathrm{ev}_{\ker\rho_{A_2}}. 
\end{align}
Let $b\in \Gamma(A_2)$. Since $\Phi\circ\Psi-\mathrm{Id}$ is valued in $\ker\rho_{A_2}$, we have $b-\Phi(\Psi(b))\in \ker\rho_{A_2}$.  This implies that $\mathrm{ev}_b-\mathrm{ev}_{\Phi(a)}\in \mathrm{ev}_{\ker\rho_{A_2}}$ with $a:=\Psi(b)\in \Gamma(A_1)$. Hence \begin{align*}
    \underbrace{\{\mathrm{ev}_{b-\Phi(a)}, \cdot\}}_{\in \mathrm{ev}_{\ker\rho_{A_2}}}=\mathcal{H}^b-\mathcal{H}^{\Phi(a)}.
\end{align*}
For $\xi\in X_{\mathcal{F}(\mathscr{G})}$, consider the integral curve $\gamma(t)=\Phi_t^{\mathcal H^b}(\xi)\in X_{\mathcal{F}(\mathscr{G})}$. We have\begin{align*}
\frac{d}{dt}\Phi^{\svee}\circ\gamma(t)&=T\Phi^\svee(\mathcal{H}^b(\gamma(t)))= T\Phi^\svee(\mathcal{H}^{\Phi(a)}(\gamma(t)))\\&\stackrel{(\ref{(5.6.ii)})}{=}\mathcal{H}^a(\Phi^\svee\circ\gamma(t)).
\end{align*}
Hence $\Phi^{\mathcal{H}^a}_t(\xi)=\Phi^\svee\circ \gamma(t)=\Phi^\svee\circ\Phi^{\mathcal{H}^b}_t(\xi)$.
\end{proof}

The following theorem is the main result of the section.

\begin{theorem}\label{thm:main1}
Let $(E\to M, \lb_E, r)$ be a Lie algebroid. 
Let $\scG\subset \Gamma(E)$ be a singular subalgebroid (Definition \ref{def:Sing-subal}) so that there is an $E$-valued almost Lie algebroid $\rho\colon A\to E$ with $\rho(\Gamma(A))=\scG$. If $\mathscr{G}$ is a coherent sheaf, then 
\begin{enumerate} 
    \item \label{(4.12.i)} the underlying space $X_{\mathcal{F}(\mathscr{G})}$ of the corresponding Poisson $C^\infty$-scheme  \[\Spec( \cF(\scG)) \hookrightarrow \Spec(C^\infty(A^*) )=(A^*, C^\infty_{A^*})\]
\begin{enumerate}
    \item is a disjoint union of symplectic smooth manifolds (referred to as \emph{symplectic leaves of\; $X_{\cF(\scG)}$}). 
    
    \item The symplectic leaves of $X_{\mathcal{F}(\mathcal G)}$ are independent of the choice of the almost Lie algebroid $A$. More precisely, for any other almost Lie algebroid $A'$, there is a canonical leaf-preserving Poisson $C^\infty$-scheme isomorphism
\[
X^{A}_{\mathcal{F}(\mathscr{G})}
\stackrel{\sim}{\longrightarrow}
X^{A'}_{\mathcal{F}(\mathscr{G})}.
\]
Here, $X^{A}_{\mathcal{F}(\mathscr{G})}$ and
$X^{A'}_{\mathcal{F}(\mathscr{G})}$ denote the corresponding realizations of
$X_{\mathcal{F}(\mathscr{G})}$ in $A^*$ and $A'^*$, respectively.
\end{enumerate}

\item \label{(4.12.ii)} For each leaf $L\subset M$ of the singular foliation $r(\mathscr{G})\subset \mathfrak X(M)$, the cotangent bundle $T^*L\hookrightarrow A^*|_L$ is a symplectic leaf of $\cF(\scG)$. 

\item \label{(4.12.iii)} for $M_{\mathrm{reg}}^{\mathscr{G}}\subset M$ the set of regular points of $\mathscr{G}$,\[X_{\mathcal{F}(\mathscr{G})}\cap A^*|_{{M}_{\mathrm{reg}}^{\mathscr{G}}}=\bigcup_{m\in M_{\mathrm{reg}}^{\mathscr{G}}}\mathrm{Im}(\rho^*_m)\supseteq \bigcup_{\substack{ m\in M_{\mathrm{reg}}^{\mathscr{G}}\cap L\\
L\in M/\mathscr{G}}}T^*_mL\] with equality when $E=TM$ and $r=\mathrm{id}$.
\end{enumerate}
\end{theorem}
\begin{remark}
Note that, in Theorem \ref{thm:main1}, even when \(M=M_{\mathrm{reg}}^\mathscr{G}\), we do not have \(X_{\mathcal{F}(\mathscr{G})}\cap A^*=A^*\) in general. See Example \ref{ex0}.
\end{remark}
\begin{remark}
Consider a Poisson manifold $(M,\pi\in\Gamma(\wedge^2 TM))$, and let
$\mathscr{G}\subset\mathfrak{X}(M)$ denote its associated symplectic
singular foliation. The latter is the image of the Lie algebroid
anchor map $\pi^\#:A=T^*M\to TM$, known as the cotangent Lie algebroid. Since $\pi$ is skew-symmetric, we have
\(
\langle \pi_x^\#(\alpha),\beta\rangle
=-\langle \alpha,\pi_x^\#(\beta)\rangle
\)
for all $x\in M$, $\alpha\in\ker(\pi_x^\#)$, and
$\beta\in T_x^*M$. That is
\(
\operatorname{Im}\big((\pi_x^\#)^*\big)
=\operatorname{Im}(\pi_x^\#).
\)
We identify $A^*$ with $TM$ so that  $X_{\mathcal{F}(\mathscr{G})}$ sits inside $TM$. Therefore, item \ref{(4.12.iii)} of Theorem \ref{thm:main1}
becomes
\[
X_{\mathcal{F}(\mathscr{G})}\cap TM|_{M_{\mathrm{reg}}}
=\bigsqcup_{x\in M_{\mathrm{reg}}}\operatorname{Im}(\pi_x^\#)
=\bigsqcup_{\substack{L\in M/\mathscr{G}\\
\text{regular}}}TL.
\]
In particular, by item \ref{(4.12.ii)} of Theorem \ref{thm:main1}, the restriction $X_{\mathcal{F}(\mathscr{G})}\cap TM|_{{M}_{\mathrm{sing}}}$ contains all the cotangent bundles of the singular leaves of $\mathscr{G}$.
\end{remark}

\begin{proof}[Proof (of Theorem \ref{thm:main1})] 

The proof of the theorem is rather lengthy; however, we opt not to divide it into sub-sections. Instead, each part of the theorem being proved is clearly emphasized in bold.\\

\noindent
\textbf{Proof of item \ref{(4.12.i)} (a)} : Let $\pi=\Pb$ be the biderivation of $C^\infty(A^*)$ defined in Remark \ref{rmk:Anchored-bundle_SF}.

Let $\mathscr V\subset \mathfrak X(A^*)$ be the $C^\infty(A^*)$-module generated by the vector fields $\mathcal{H}^a, \mathcal{H}^{f}$ for $ a\in \Gamma(A), f\in C^\infty(M)$.\\

\noindent
\textbf{Claim 1}: The vector fields of $\mathscr{V}$ preserve $\langle \eta_A(\ker\rho )\rangle$, i.e., ${v}[\langle \eta_A(\ker\rho )\rangle]\subset \langle \eta_A(\ker\rho )\rangle$ for all $v\in \mathscr{V}$.
\begin{proof}[Proof of (Claim 1)]
    Indeed, for $a\in \Gamma(A), f\in C^\infty(M)$ and $\xi\in \ker \rho$ we have \begin{align*}
    \mathcal{H}^a[\eta_A(\xi)]&=\{\eta_A(a), \eta_A(\xi)\}=\eta_{[a, \xi]_A}\in \langle \eta_A(\ker\rho )\rangle.\\\\\mathcal{H}^{f}[\eta_A(\xi)]&=\{\mathrm{pr}^*f, \eta_A(\xi)\}=-\underbrace{r\circ\rho(\xi)}_{=0}[f]\circ \mathrm{pr}=0.
\end{align*}
The condition $\eta_{[a, \xi]_A}\in \langle \eta_A(\ker\rho )\rangle$ holds because $\rho([a, \xi]_A)=[\rho(a), \underbrace{\rho(\xi)}_{=0}]_E=0$. This proves the claim.
\end{proof}

Now, since $\ker\rho$ is locally finitely generated, so is the ideal $\langle \eta_A(\ker\rho )\rangle$. By Lemma \ref{lem:flow-zero-locus}, the flow $\Phi^{\mathcal H^{\nu}}_t$ preserves ($\nu= a$ or $\nu= \mathrm{pr}^*f$ for some $a\in\Gamma(A), f\in C^\infty(M)$) the zero locus of the ideal $\langle \eta_A(\ker\rho )\rangle$. By Nagano–Sussmann (Theorem \ref{thm:Nagano–Sussman}), for $\alpha \in A^*$, the set 
\begin{equation}\label{eq:NaganoSuss}
        \mathcal{O}_\alpha:=\left\{\Phi_{t_1}^{\mathcal{H}^{\nu_{1}}}\circ \Phi_{t_2}^{\mathcal{H}^{\nu_{2}}}\circ \cdots \circ \Phi_{t_k}^{\mathcal{H}^{\nu_{k}}}(\alpha),\; t_i\in \mathbb R, i=1,\ldots,k\in \mathbb{N}\right\}\subset A^*
    \end{equation}
 is an immersed submanifold. They form a partition of $A^*$. We implicitly assume that the flows are defined. 

 Since the flows of the vector fields $\mathcal{H}^{\nu_{i}}$'s preserve the zero locus $ Z_{\langle \eta_A(\ker\rho )\rangle}=X_{\mathcal{F}(\mathscr{G})}$ of the ideal $\langle \eta_A(\ker\rho )\rangle$, for every $\alpha\in X_{\mathcal{F}(\mathscr{G})}$ we have $\mathcal{O}_\alpha\subset X_{\mathcal{F}(\mathscr{G})}$.
\\

\noindent
\textbf{Claim 2}: The manifolds $\mathcal{O}\equiv\mathcal{O}_\alpha$ are symplectic for all $\alpha\in  X_{\mathcal F(\mathscr{G})
} =Z_{\langle \eta_A(\ker\rho )\rangle}$.

\begin{proof}[Proof of (Claim 2)]
    To start with, the restriction $(\mathcal{O}, \pi|_\mathcal{O})$ of the bivector field $\pi$ to $\mathcal{O}$ is well-defined and is a Poisson manifold (since  $(i)$ the vector fields $\pi(dF)=\{F, \cdot\}$ are tangent to $\mathcal{O}$ for $F\in C^\infty(A^*)$ $(ii)$ the Jacobiator vanishes on $\mathcal{O}\subset Z_{\langle \eta_A(\ker\rho )\rangle}$). Consider the vector bundle map  $\pi^\#\colon T^*A^*\to TA^*, \nu\mapsto \iota_\nu\pi$ and its restriction \[\pi^\#|_{\mathcal{O}}\colon T^*\mathcal{O}\to T\mathcal{O}\] to the orbit $\mathcal{O}$. For $\xi_x\in \mathcal{O}\subset Z_{\langle \eta_A(\ker\rho )\rangle}$ the tangent space of $\mathcal{O}$ at $\xi_x$ is \[\mathrm{Span}\bigr\{\phi_*(\mathcal{H}^{\nu})(\xi_x)\mid \nu= a\quad \text{or}\quad \nu= \mathrm{pr}^*f \quad \text{for}\quad a\in\Gamma(A), f\in C^\infty(M)\bigr\}\subset T_{\xi_x}A^* \]where $\phi$ is of the form $\Phi_{t_1}^{\mathcal{H}^{\nu_{1}}}\circ \Phi_{t_2}^{\mathcal{H}^{\nu_{2}}}\circ \cdots \circ \Phi_{t_k}^{\mathcal{H}^{\nu_{k}}}$. By Lemma \ref{lem:Pushforward-of-hamiltonian}, the push forward  of the vector field $\mathcal{H}^{\nu}$ through $\phi$ applied to $\xi_x$ is the evaluation  of a vector field of the form $\{F, \cdot\}$ for some $F\in C^\infty(A^*)$. Therefore\[\im (\pi^\#_{\xi_x})=T_{\xi_x}\mathcal{O}.\]
Hence, $\pi|_\mathcal{O}$ is a regular Poisson bivector field whose rank is $\dim \mathcal O$. Thus, there exists a unique symplectic structure $\omega_\mathcal{O}$ on $\mathcal{O}$, where $\pi|_\mathcal{O}$ is the canonical Poisson structure associated with~$\omega_\mathcal O$.
\end{proof}This proves Item \ref{(4.12.i)} (a).\\

\noindent
\textbf{Proof of item \ref{(4.12.i)} (b)} : By Proposition \ref{prop:equivalence} and Corollary \ref{cor:equivalence}, there exist two vector bundle maps
\[
    \xymatrix{
    A_2^* \ar@<2pt>[r]^{\Phi^\svee} &
    A_1^* \ar@<2pt>[l]^{\Psi^\svee}
    }
\]
such that the induced pullback maps
\(
    \xymatrix{
    C^\infty(A_1^*) \ar@<2pt>[r]^{(\Phi^\svee)^*} &
    C^\infty(A_2^*) \ar@<2pt>[l]^{(\Psi^\svee)^*}
    }
\)
satisfy
\begin{align}
    \label{eq:A}
    &(\Phi^\svee)^*(\mathrm{ev}_a)
    = \mathrm{ev}_{\Phi(a)},\quad\text{for}\; a\in \Gamma(A_1);&\\
    &(\Phi^\svee)^*(\mathrm{ev}_{\ker\rho_{A_1}})
    \subset \mathrm{ev}_{\ker\rho_{A_2}};&\\
    \label{eq:Poisson_iso}
    &(\Phi^\svee)^*(\mathrm{ev}_{[a,b]_{A_1}})
    - \mathrm{ev}_{[\Phi(a),\Phi(b)]_{A_2}}
    \in \mathrm{ev}_{\ker\rho_{A_2}};&\\&(\Phi^\svee)^*(\mathrm{pr}_{1}^*(r\circ\rho_{A_1}(a)f))
    = \mathrm{pr}^*_{2}(r\circ\rho_{A_2}(\Phi(a))f), 
    \quad \text{for}\; a\in \Gamma(A_1),\, f\in C^\infty(M);&\\
    \label{eq:B}
    &(\Phi^\svee)^*\circ(\Psi^\svee)^*-\mathrm{Id}
    \quad \text{is $(\mathrm{ev}_{\ker\rho_{A_2}})$-valued}.&
\end{align}
Analogous relations hold for $(\Psi^\svee)^*$. These imply that
\begin{itemize}
    \item the maps $\Phi^\svee$ and $\Psi^\svee$ send
    $X_{\mathcal{F}(\mathscr{G})}^{A_2}$ and
    $X_{\mathcal{F}(\mathscr{G})}^{A_1}$ to each other, respectively;
\item  the maps $(\Phi^\svee)^*$ and $(\Psi^\svee)^*$ induce mutually inverse Poisson $C^\infty$-ring isomorphims\[
    \xymatrix{
    C^\infty(A_1^*)/\langle \mathrm{ev}_{\ker\rho_{A_1}}\rangle \ar@<2pt>[r] &
    C^\infty(A_2^*)/\langle \mathrm{ev}_{\ker\rho_{A_2}}\rangle \ar@<2pt>[l]
    };
\]
    \item the restrictions of $\Phi^\svee$ and $\Psi^\svee$ to
    $X_{\mathcal{F}(\mathscr{G})}^{A_2}$ and
    $X_{\mathcal{F}(\mathscr{G})}^{A_1}$, respectively, are leaf-preserving (by item \ref{item:leaf-preserving} of Lemma \ref{lem:leaf-preserving});

    \item these restrictions are inverse to each other.
\end{itemize}

\vspace{0.3cm}

\noindent
\textbf{Proof of item \ref{(4.12.ii)}} : Let $\rho':=r\circ\rho$ and  $L\subset M$ be a leaf of the singular foliation $\rho'(\Gamma(A))=r(\mathscr{G})\subset \mathfrak X(M)$.  The restriction of the anchor map $\rho'$ to $L$ induces a surjective vector bundle morphism\[\rho'|_L\colon A|_L\to TL.\] The latter comes equipped with an almost Lie algebroid structure, in particular it is bracket preserving. The dual map is an immersion $i\colon T^*L\hookrightarrow A^*, \alpha_x\mapsto \rho'^*_x(\alpha_x)$.

We need to check that $T^*L\subset  Z_{\langle \eta_A(\ker\rho )\rangle}$: let $\alpha_m\in T^*_mL$ and $\xi\in\ker\rho$\begin{align*}
    \langle \rho'^*_m(\alpha_m), \xi(m)\rangle&=\langle \rho^*_m\circ r^*_m(\alpha_m), \xi(m)\rangle\\&=\alpha_m\circ r_m\circ \rho_m(\xi(m))=0.
\end{align*}

\noindent
\textbf{Claim 3}. The vector fields $\mathcal{H}^{\nu}\in \mathfrak X(A^*)$ ($\nu= a\quad \text{or}\quad \nu= \mathrm{pr}^*f \quad \text{for}\quad a\in\Gamma(A), f\in C^\infty(M)$) are tangent  $T^*L$ for $L\in M/ \mathcal G$, and their flows preserves $T^*L$  at all time.

\begin{proof}[Proof of (Claim 3)] Let $L\in M/
\mathcal G$ be a leaf of the singular subalgebroid $\mathscr{G}$. Let $f\in C^\infty(M)$ and consider the vector field $\mathcal H^{f}=\{\mathrm{pr}^*f, \cdot\}\in \mathfrak X(A^*)$. Since for all $g\in C^\infty(M)$, $\{\mathrm{pr}^*f, \mathrm{pr}^*g\}=0$, the vector field $\mathcal{H}^f$ is $\mathrm{pr}$-projectable to zero, i.e., vertical. On the other hand, for every section $a\in \Gamma(A)$ and $\xi\in  T^*_mL\hookrightarrow A^*_m$, \begin{align*}
    \underbrace{\rho'^*_m(d_mf|_L)}_{\in\, \mathrm{im}(\rho'^*_m)}(a(m))&=\langle d_mf|_L, \rho'(a(m)\rangle=\rho'(a)[f](m)\\&=\{\mathrm{ev}_a, \mathrm{pr}^*f\}(\xi)=\mathcal{H}^f[\mathrm{ev}_a](\xi).  
\end{align*}
Hence, $\mathcal{H}^f$ is tangent to $T^*L$. In addition, $\mathcal{H}^f$ is constant along of the fibers of $T^*L$, thus $\Phi^{\mathcal{H}^f}_t(T^*L)=T^*L$ for all $t$ where it makes sense.\\

\noindent
Note that for  $a\in \Gamma(A)$, the vector field $\mathcal{H}^a=\{\mathrm{ev}_a, \cdot\}\in \mathfrak X(A^*)$ associated with the evaluation function $\mathrm{ev}_a\colon A^*\to \mathbb R$  is a linear vector field on $A^*$ which projects to $\rho'(a)$, that is, the following diagram is a vector bundle morphism $$\xymatrix{A^*\ar[r]^{\mathcal{H}^a} \ar[d]_{\mathrm{pr}}&TA^*\ar[d]^{T\mathrm{pr}}\\M\ar[r]^{\rho'(a)}&TM}$$ In particular, the flow $\phi^{\mathcal{H}^a}_t\colon A^*\to A^*$ of the vector field $\mathcal{H}^a$ is an isomorphism of vector bundles.\\

Now, the proof proceeds in the following two steps.
\begin{enumerate}
    \item [(a)] For $a\in \Gamma(A)$ with $X=\rho'(a)$, the flow $\phi^{\mathcal{H}^a}_t\colon A^*\to A^*$ of the vector field $\mathcal{H}^a$ on $A^*$ exists whenever the flow $\Phi^X_t$ exists and the following diagram commutes \begin{equation}\label{eq:diag1}
\xymatrix{T^*M\ar[r]^{(T\Phi_t^X)^*}\ar[d]_{\rho'^*}&T^*M\ar[d]^{\rho'^*}\\ A^*\ar[d]\ar[r]^{\Phi_t^{\mathcal{H}^a}}&A^*\ar[d]^{}\\M\ar[r]^{\Phi_t^X} &M.}
\end{equation}
 To see that, let $\mathcal{H}_X$ denotes the Hamiltonian vector field associated with $X\in \mathfrak X(M)$ w.r.t the canonical Poisson structure of $T^*M$. The vector field $(\mathcal{H}_X+\mathcal{H}^a, X)$ is a linear vector field on $T^*M\oplus A^*$ that is tangent to the graph of $\rho'^*\colon T^*M\to A^*$$$ \left\{ (\xi, \rho'^*(\xi) ) \,|\,  \xi \in T^*M \right\} \subset T^*M \oplus A^*. $$  Hence, the diagram \eqref{eq:diag1} commutes. The complete argument is very similar to \cite[Proposition 1.8]{Lo1}.  

\item [(b)] For every $m\in M$, the flow $\Phi_t^{\mathcal{H}^a}$ (whenever it makes sense) sends $\mathrm{Im}{\rho'^*_m}$ to  $\mathrm{Im}\rho'^*_{\Phi^X_t(m)}$, that is
$$\Phi_t^{\mathcal{H}^a}\left(\mathrm{Im}{\rho'^*_m}\right)=\mathrm{Im}\rho'^*_{\Phi^X_t(m)}$$in view of Diagram \eqref{eq:diag1}. In particular,  for every $m\in L\in M/\mathscr{G}$\[\Phi_t^{\mathcal{H}^a}\left(T^*_mL\right)=T^*_{\Phi^X_t(m)}L.\] That is, \(\Phi_t^{\mathcal{H}^a}\left(T^*L\right)=T^*L.\)

\end{enumerate}
\end{proof}
It follows that for $\xi\in T_m^*L$, $\mathcal{O}_\xi\subset T^*L$. In particular,
\(
\operatorname{im}\pi_\xi^\#\subseteq T_\xi(T^*L).
\)
Let $k=\dim L$. Choose $a_1,\ldots,a_k\in\Gamma(A)$ such that
$\rho'(a_1)(m),\ldots,\rho'(a_k)(m)$ form a basis of $T_mL$, and choose
$f_1,\ldots,f_k\in C^\infty(M)$ such that
\[
d_mf_j\bigl(\rho'(a_i)(m)\bigr)=\delta_{ij}.
\]
Then
\[
\mathcal H^{a_1}(\xi),\ldots,\mathcal H^{a_k}(\xi),
\mathcal H^{f_1}(\xi),\ldots,\mathcal H^{f_k}(\xi)
\]
are linearly independent. Indeed, the $\mathcal H^{a_i}$ project to the basis
$\rho'(a_i)(m)$ of $T_mL$, whereas the $\mathcal H^{f_j}$ are vertical and satisfy
\[
\mathcal H^{f_j}[\operatorname{ev}_{a_i}](\xi)=-\delta_{ij}.
\]
Therefore,
\(
\dim\operatorname{im}\pi_\xi^\#\geq 2k.
\)
Since
\(
\dim T_\xi(T^*L)=2k,
\)
the above inclusion is an equality:
\[
{\operatorname{im}\pi_\xi^\#=T_\xi(T^*L).}
\]
Since $L$ is connected, so is $T^*L$. This implies that $\mathcal{O}_\xi$ is open in  $T^*L$, thus $T^*L$ is a disjoint union of open sets. This forces $\mathcal{O}_\xi=T^*L$. Therefore, $\bigr(T^*L, (i^{-1})^*\omega_{can}\bigr)\hookrightarrow X_{\mathcal{F}(\mathscr{G})}\subset A^*$ is a symplectic leaf of  $X_{\mathcal{F}(\mathscr{G})}$, where $\omega_{can}$ is the canonical symplectic form on $T^*L$. This proves Item \ref{(4.12.ii)}.

Finally, we check that \[X_{\mathcal{F}(\mathscr{G})}\cap A^*|_{{M}_{\mathrm{reg}}^\mathscr{G}}= \bigsqcup_{m\in M_{\mathrm{reg}}^\mathscr{G}}\mathrm{Im}(\rho^*_m).\]Let $m\in M_{\mathrm{reg}}^\mathscr{G}$ and  $\alpha\in X_{\mathcal{F}(\mathcal G)}\cap A^*_m$. There is an open neighborhood $\mathcal{U}$ of $m$ such that $\ker\rho|_\mathcal{U}\subset \Gamma(A)|_\mathcal{U}$ is a vector bundle over $\mathcal U$. Every $u\in \ker\rho_m$ admits a (local) extension $\widetilde{u}\in \ker \rho|_\mathcal{U}$. We have \[\langle \alpha, u\rangle=\langle \alpha, \widetilde{u}(m)\rangle=\mathrm{ev}_{\widetilde{u}}(\alpha)=0\quad \text{by definition of $\alpha$}.\]This implies that $\alpha\in (\ker\rho_m)^\circ=\mathrm{Im}(\rho_m^*)$. The other inclusion is immediate. The incution \[\bigcup_{m\in M_{\mathrm{reg}}^{\mathscr{G}}}\mathrm{Im}(\rho^*_m)\supseteq \bigcup_{\substack{ m\in M_{\mathrm{reg}}^{\mathscr{G}}\cap L\\
L\in M/\mathscr{G}}}T^*_mL\] follows from the relation $\mathrm{Im}(\rho'^*)=\mathrm{Im}(r\circ\rho)^*=\mathrm{Im}(\rho^*\circ r^*)\subseteq \mathrm{Im}(\rho^*)$.
\end{proof}
We end the section with some examples.
\begin{example}\label{ex0} Let $\mathscr{G}=\{0\}$ be the trivial singular subalgebroid. The corresponding Poisson $C^\infty$-ring is $\mathcal{F}(\mathscr{G})\simeq C^\infty(M)$, with zero Poisson bracket. We have \[ \Spec(\mathcal F(\mathcal G))\cong \Spec(C^\infty(M))=(M, C^\infty_M),\]  and  every point of \(M\) is a zero-dimensional symplectic leaf. 
\end{example}
\begin{example}\label{ex1}
Let $\mathscr{G}\subset \mathfrak X(M)$ be a projective\footnote{That is, $\mathscr{G}$ is a finitely generated projective $C^\infty(M)$-module.} singular foliation so that there is a vector bundle $A$ with $\Gamma(A)\simeq \mathscr{G}$. In particular, there is a Lie algebroid structure on $A$ whose anchor map is injective on an open dense subset. Then $\Spec (\mathcal{F}(\mathscr{G}))=\Spec(C^\infty(A^*))=(A^*, C^\infty_{A^*})$. The symplectic leaves of $X_{\mathcal{F}(\mathscr{G})}=A^*$ are the leaves of the linear Poisson structure on $A^*$.
\end{example}
\begin{example}\label{ex2}
Let $\mathcal G$ be the singular foliation on $M=\mathbb R^3$ generated by the image of the cotangent Lie algebroid $\rho\colon A=T^*\mathbb R^3\to T\mathbb R^3$,  $\rho(dx)= z\frac{\partial}{\partial y}-y\frac{\partial}{\partial z}$,\, $\rho(dy)=x\frac{\partial}{\partial z}-z\frac{\partial}{\partial x}$,\, $\rho(dz)= y\frac{\partial}{\partial x}-x\frac{\partial}{\partial y}$. The latter can be seen as the induced singular foliation of the coadjoint representation of $\mathfrak{so}(3)$. The kernel of $\rho$ is globally generated by the section $\theta=xdx+ydy+zdz$, i.e.,  \[ \ker\rho = C^\infty(\mathbb{R}^3)\theta. \] In particular, \(\ker\rho\) is globally finitely generated. Note that $\theta$ is in the center of the bracket of $A$.

By using coordinates $(x,y,z;u,v,w)$ on \(T\mathbb{R}^3=\mathbb{R}^3\times\mathbb{R}^3\), where \( u=\operatorname{ev}_{dx},\; v=\operatorname{ev}_{dy}, \; w=\operatorname{ev}_{dz}. \) The fiberwise linear function corresponding to \(\theta\) is \(\eta_A(\theta)=xu+yv+zw. \) Therefore \[ \cF(\mathscr{G}) = \frac{C^\infty(T\mathbb{R}^3)} {\langle xu+yv+zw\rangle}.\] We denote by $\bar F$ the projection of an element $F\in C^\infty(T\mathbb{R}^3)$ onto $\mathcal{F}(\mathcal G)$. The Poisson structure is determined by the brackets \begin{align*}
    &\{\bar u, \bar v\}:=\bar w \;; \{\bar u, \bar w\}=-\bar v \;; \{\bar v, \bar w\}:=\bar u && \text{(relation \eqref{eq:Rinehart-Poisson1})}\\&\{\bar u, \bar y\}:=\bar z\;;\{\bar u, \bar z\}:=-\bar y\;; \{\bar v, \bar x\}:=-\bar z \;;\{\bar v, \bar z\}:=\bar x \;;\{\bar w, \bar x\}:=\bar y\;;\{\bar w, \bar y\}:=-\bar x && \text{(relation \eqref{eq:Rinehart-Poisson2})}\\&\{\bar u, \bar x\}=\{\bar v, \bar y\}=\{\bar w, \bar z\}=\{\bar x, \bar y\}=\{\bar x, \bar z\}=\{\bar y, \bar z\}:=0&& \text{(relation \eqref{eq:Rinehart-Poisson3})}.
\end{align*}
The underlying space of the Poisson $C^\infty$-scheme $\Spec(\mathcal F(\mathcal G))$ is \begin{align*}
    X_{\cF(\mathscr{G})}&= \left\{ (x,y,z;u,v,w)\in T\mathbb{R}^3 \mid xu+yv+zw=0 \right\}\subset \mathbb{R}^6\\&=(\{0\}\times \mathbb R^3) \bigsqcup_{(x, y, z)\neq 0}T_{\frac{(x, y, z)}{\lVert (x,y,z) \rVert}}^*S^2.
\end{align*}
The symplectic leaves decomposition of $X_{\cF(\mathscr{G})}$ are exactly the symplectic leaves of the linear Poisson structure on $T\mathbb{R}^3$ associated with the cotangent Lie algebroid $T^*\mathbb{R}^3\to T\mathbb R^3$ containing a point of $X_{\cF(\mathscr{G})}$. In particular,
\begin{itemize}
    \item $\{0\}$ is a symplectic leaf; more generally $T^*S^2_r$ is a symplectic leaf, where $S^2_r$ denotes the 2-sphere of radius $r\in \mathbb{R}_{+}^*$. Equivalently, since the functions $C_1=x^2+y^2+z^2$ and $C_2=\eta_A(\theta)=xu+yv+zw$ are Casimir and $C=(C_1, C_2)\colon \mathbb{R}^6\to \mathbb{R}^2$ is a submersion outside of $\{0\}\times \mathbb{R}^3$, all the 4-dimensional symplectic leaves of $X_{\cF(\mathscr{G})}$ are exactly the level sets $C^{-1}(r^2, 0)$, $r>0$.
    \item the 2-dimensional symplectic leaves  of $X_{\cF(\mathscr{G})}$ are  concentric spheres (except zero) in the fiber $\{0\}\times \mathbb{R}^3_{u, v, w}\simeq \mathfrak{so}(3)^*$.\end{itemize}

\end{example}

\begin{example}\label{ex:vanishing_order2}
Let $\mathscr{G} \subset \mathfrak{X}(\mathbb{R}^2)$ be the singular foliation generated by vector fields on $\mathbb{R}^2$ that vanish to second order. Specifically, $\mathscr{G}$ is globally generated by the vector fields
\[x^2\partial_x,\, y^2\partial_x,\, xy\partial_x,\, x^2\partial_y,\, y^2\partial_y,\, xy\partial_y.\]
Define the anchored bundle
\[
\rho:  A=\mathbb{R}^2\times \mathbb{R}^6 \longrightarrow T\mathbb{R}
\]
by sending the standard frame \(e_1, \dots, e_6\) to
\[
\tau(e_1) = x^2 \partial_x, \quad
\tau(e_2) = xy \partial_x, \quad
\tau(e_3) = y^2 \partial_x,
\]
\[
\tau(e_4) = x^2 \partial_y, \quad
\tau(e_5) = xy \partial_y, \quad
\tau(e_6) = y^2 \partial_y.
\]
The  kernel of $\rho\colon \Gamma(A)\to \mathfrak X(\mathbb{R}^2)$ is given by
\[
\ker(\rho) =\operatorname{Span}_{C^\infty(\mathbb{R}^2)} \{ x e_2-y e_1,\; x e_3-y e_2,\; x e_5-y e_4 ,\; x e_6-y e_5 \}.
\]

Let $\xi_i,
,i=1,\ldots,6$ be the standard linear coordinates on $\mathbb{R}^6$. The corresponding $C^\infty$-ring of $\mathscr{G}$ is \[\cF(\mathscr{G})=C^\infty(\mathbb R^2\times \mathbb R^6)/\langle\, x \xi_2-y \xi_1,\; x \xi_3-y \xi_2,\; x \xi_5-y \xi_4 ,\; x \xi_6-y \xi_5 \,\rangle.\]  
We have used the identification $A^*_{(x, y)}\simeq \mathbb{R}^6$. See \cite[Example 6.16]{LL} for the description of the $C^\infty$-Poisson structure on $\mathcal{F}(\mathscr{G})$. The underlying space of the Poisson $C^\infty$-scheme $\Spec(\mathcal F(\mathcal G))$ is\[X_{\mathcal{F}(\mathscr{G})}=A^*_{(0,0)}\cup \bigr(T^*(\mathbb{R}^2\setminus \{(0,0)\})\bigr)=\mathbb{R}^6\cup \bigr(\mathbb{R}^2\setminus \{(0,0)\}\times \mathbb{R}^2\bigr)\hookrightarrow \mathbb{R}^8.\]
The symplectic leaves of $X_{\mathcal{F}(\mathscr{G})}$ consist of the singletons in $\mathbb{R}^6 \hookrightarrow \mathbb{R}^2 \times \mathbb{R}^6$ and the 4-dimensional leaf $(\mathbb{R}^2 \setminus \{(0,0)\}) \times \mathbb{R}^2$.
\end{example}

The conclusion of Theorem~\ref{thm:main1} may still hold without the assumption of (local) finite generation, as shown by the following example.
\begin{example}
    Consider the singular foliation
\[
\mathcal G
=
\left\langle
\chi(x)\partial_x
\right\rangle\subset \mathfrak X(\mathbb R),
\qquad
\chi(x)=
\begin{cases}
e^{-1/x^2}, & x>0,\\
0, & x\le 0.
\end{cases}
\] Let $A=\mathbb R\times\mathbb R$ be the trivial line bundle with global frame $e\equiv 1$ and anchor
\[
\rho(e)=\chi(x)\partial_x.
\]
Here $\ker\rho$ is not locally finitely generated, see \cite[Example 3.38]{LLS}. The corresponding Poisson $C^\infty$-ring of $\mathcal G$ is \[\mathcal{F}(\mathscr{G})=C^\infty(\mathbb R^2)/\langle f(x)\xi, f|_{[0, \infty)}\equiv 0\rangle\] where $x, \xi$ are the global coordinates on $\mathbb{R}^2=\mathbb{R}\times \mathbb R$. The Poisson structure is induced by the biderivation \[\{\xi, x\}=\chi
(x).\]Note that $\mathscr{G}$ does arise from a Lie algebroid structure on $A$ \cite{LLS}. This can also be seen without any calculations: the linear Poisson structure on $\mathcal{F}(\mathscr{G})$ extends to $C^\infty(\mathbb{R}^2)$ for dimensional reasons.

The underlying space of the Poisson $C^\infty$-scheme $\Spec(\mathcal F(\mathcal G))$ is \[X_{\mathcal{F(\mathcal G)}}=(\mathbb{R}^{\times}_{-}\times \{0\})\cup (\mathbb{R}_{+}\times \mathbb{R})\subset \mathbb{R}^2.\]
The symplectic leaf decomposition of $X_{\mathcal{F}(\mathcal G)}$ consists of the $0$-dimensional leaves $\{(0, b)\}$, $\{(a,0)\}$, for $a<0$ and $b\in \mathbb{R}$, together with the $2$-dimensional leaf $\mathbb{R}^{\times}_{+}\times\mathbb{R}\subset\mathbb{R}^2$.

\end{example}

\section{Helffer-Nourrigat cone of a singular subalgebroid}\label{sec:6}
As we have seen in the previous section, every singular subalgebroid $\mathscr{G}$ canonically determines a Poisson $C^\infty$-ring $\mathcal{F}(\mathscr{G})$. Moreover, if $\mathscr{G}$ arises as the image of an anchored bundle $A \to E$, then there is a closed embedding
\[
\Spec(\mathcal{F}(\mathscr{G})) \hookrightarrow \Spec(C^\infty(A^*)) = (A^*, C^\infty_{A^*}),
\]
and the underlying space $X_{\mathcal{F}(\mathscr{G})}$ of $\Spec(\mathcal{F}(\mathscr{G}))$ is a disjoint union of immersed symplectic submanifolds of $A^*$.

On the other hand, there is another canonical Poisson $C^\infty$-ring naturally associated with $\scG$, namely the Helffer-Nourrigat cone \cite{HN,AMY}, which arises in the study of longitudinal differential operators on singular foliations. The Helffer–Nourrigat cone of a singular foliation admits a characterization in terms of anchored vector bundles and the Nash blowup\cite{Lo1,Lo2}, as presented in \cite{Lo3}. In this section, we compare the two resulting Poisson $C^\infty$-ring structures. 

In what follows, we assume that 
our singular subalgebroid $\scG$ is finitely generated so that there is an anchored bundle $\rho\colon A\to E$ over $\scG$, i.e.,  such that $\rho(\Gamma(A))=\scG$. We shall also denote by $\rho'\colon A\to TM$ the composition $r\circ\rho$. Here $r\colon E\to TM$ is the anchor map of the Lie algebroid $E$.

\begin{definition}
Let $\rho\colon A\to E$ be an anchored bundle over $\scG$ and $ \rho^* \colon  E^* \to A^*$ the dual of $\rho$. The \emph{\underline{basic} Helffer-Nourrigat cone of  $\scG$}, is the closed subset of $ A^*$ given by 
 \[{\mathrm{HN}}(\scG):= \overline{ \cup_{m\in M_{\mathrm{reg}}} \rho_m'^* (T_m^*M)}\subset A^*.\]We shall denote by $\mathfrak P\colon{\mathrm{HN}}(\scG) \to M $ the restriction of the projection $A^* \to M$ to the ${\mathrm{HN}}(\scG)$.
\end{definition}

\begin{remark}
The basic Helffer-Nourrigat cone of $\mathscr{G}$ is in fact the Helffer-Nourrigat cone (in the sense of \cite{AMY,Lo3}) of the basic singular foliation $r\circ \rho(\Gamma(A))\subset \mathfrak X(M)$ of $\mathscr{G}$, which justifies the adjective ``basic.'' One could, of course, define a cone that genuinely depends on $\mathscr{G}$ (by replacing $\rho_m'^* (T_m^*M)$ by $\rho_m^* (E_m^*)$) on  $M_{\mathrm{reg}}^\mathscr{G}$ rather than merely on its basic singular foliation, but we do not pursue that here. In particular, when $E=TM$ and $r=\mathrm{id}_{TM}$, so that $\mathscr{G}\subset \mathfrak X(M)$ is a singular foliation, the definition recovers the usual notion of the Helffer-Nourrigat cone of $\mathscr{G}$. In what follows, we shall omit the term ``basic.''
\end{remark}
\begin{remark}\label{rmk:cone}
The closed subset ${\mathrm{HN}}(\scG)\subset A^*$ is indeed a cone, as it is, by construction, stable by multiplication by a scalar.
\end{remark}

In the remainder of this section, we will focus on singular foliations, that is when $E=TM$ and $\mathscr{G}\subset \mathfrak X(M)$ is a singular foliation.
\begin{proposition}[\cite{Lo3}]Let $\mathfrak P\colon {\mathrm{HN}}(\mathcal F) \to M $ be the Helffer-Nourrigat cone of a singular foliation associated with an anchored bundle $\rho\colon A\to TM$.
\begin{enumerate}
\item For every $m \in M$, the fiber of $\mathfrak P$ is non-empty, and is contained in the annihilator of \[\mathrm{Sker}_m(\rho)=\{u\in A_m\mid \exists\; \text{local section $\widetilde{u}$ through $u$ with}\; \rho(\widetilde{u})=0\}\subseteq \ker\rho_m.\]
\item  $\mathfrak P^{-1}(m)={\mathrm{Im}}(\rho^*_m)={\mathrm{ker}}(\rho_m)^{\circ}$, for every\footnote{We use the symbol $ {}^\circ$ for the annihilator.} $m \in M_{\mathrm{reg}}$.
\end{enumerate}
\end{proposition}

\begin{proposition}[\cite{Lo3}]
  The {Helffer-Nourrigat cone of $\scG$} does not depend on the choice of an anchored bundle $A\to TM$.
\end{proposition}

Now, here is the main statement of this section.
\begin{theorem}\label{thm:HNcone}
Let $\mathfrak P\colon {\mathrm{HN}}(\scG)\to M$ be the Helffer-Nourrigat cone of a singular foliation $\scG$ associated with an anchored bundle $\rho\colon A\to TM$. Let $\mathcal I_{\mathrm{HN}(\scG)}\subset C^\infty(A^*)$ be the  ideal of smooth
functions on $A^*$ that vanish on $\mathrm{HN}(\scG)
\subset A^*$.  Then

\begin{enumerate}
\item\label{item:5.12.iii} $\cin(\mathrm{HN}(\scG)):=C^\infty(A^*)/\mathcal{I}_{\mathrm{HN}(\scG)}$ is a Poisson  $C^\infty$-ring\footnote{$\cin(\mathrm{HN}(\scG))$ corresponds to the $C^\infty$-ring of all Whitney
smooth functions on $\mathrm{HN}(\scG)\subset A^*$.} and there is a closed embedding
\[
\Spec(\cin(\mathrm{HN}(\scG)))\hookrightarrow \mathrm{Spec}(\cF(\scG))
\]
of affine Poisson $C^\infty$-schemes. Moreover,  the underlying space of  $\Spec(C^\infty(A^*)/\mathcal{I}_{\mathrm{HN}(\scG)})$ is $\mathrm{HN}(\scG)$;
\item  \label{item:union-by-smooth-manifolds} the space $\mathrm{HN}(\scG)$ is included in  $ X_{\mathcal{F}(\mathscr{G})}\subset A^*$ and is a disjoint union of symplectic manifolds.
\end{enumerate}
\end{theorem}

\begin{proof}
We have \[\mathcal I_{\mathrm{HN}(\scG)}=\{F\in C^\infty(A^*)\mid F(\xi)=0, \forall \xi\in \ker \rho_x^\circ, x\in M_{\mathrm{reg}}\}.\] For every \(a\in\ker\rho\) and every
\(\xi\in(\ker\rho_x)^\circ\) and $x\in M_{\mathrm{reg}}$,
\[
\eta_A(a)(\xi)
=\langle \xi,a(x)\rangle
=0,
\]
so every generator of the ideal
\(\langle \eta_A(\ker\rho)\rangle\)
vanishes on \(\mathrm{HN}(\scG)\). Therefore,
\[
\langle \eta_A(\ker\rho)\rangle
\subset \mathcal I_{\mathrm{HN}(\scG)}.
\]There is a surjective map $\varphi\colon \cF(\scG)\longrightarrow C^\infty(A^*)/\mathcal I_{\mathrm{HN}(\scG)}$. By Lemma \ref{lem:A4},
\[
\Spec(\varphi): \Spec(C^\infty(A^*)/\mathcal I_{\mathrm{HN}(\scG)}) \to \Spec(\cF(\scG))=(A^*, C^\infty_{A^*})\]
is a closed embedding.

Since $\mathrm{HN}(\scG)\subset A^*$ a closed subset, the underlying space of $\cin(\mathrm{HN}(\scG)=C^\infty(A^*)/\mathcal I_{\mathrm{HN}(\scG)}$ is $\mathrm{HN}(\scG)=Z_{\mathcal{I}(\mathrm{HN}(\scG))}\subset X_{\mathcal{F}(\mathscr{G})}$, see Lemma \ref{lem:A1}. An alternative way to see that is by using Theorem \ref{thm:main1} (item \ref{(4.12.iii)}):  By definition of $\mathrm{HN}(\scG)$, we have $$\mathrm{HN}(\scG)\cap A^*|_{M_{\mathrm{reg}}}=\bigsqcup_{\begin{array}{c}
L\in M/\mathscr{G}\\\mathrm{regular}
\end{array}} T^*L\overset{\text{Theorem \ref{thm:main1}}}{=}X_{\mathcal{F}(\mathscr{G})}\cap A^*|_{{M}_{\mathrm{reg}}}\subset X_{\mathcal{F}(\mathscr{G})}.$$
It follows that $\mathrm{HN}(\scG)\subset X_{\mathscr{G}(\mathcal{F})}$. Since the almost hamiltonian vector fields of $\Pb_{A^*}$ preserve the cotangent bundles of the (regular) leaves of $\mathscr{G}$, $\mathrm{HN}(\scG)$ is a disjoint union of symplectic leaves.\\

We need to show that $\mathcal I_{\mathrm{HN}(\scG)}$ is an $\Pb_{A^*}$-ideal, i.e., $\{C^\infty(A^*), I_{\mathrm{HN}(\scG)}\}_{A^*}\subset I_{\mathrm{HN}(\scG)}$. 

Let $F\in \mathcal I_{\mathrm{HN}(\scG)},G\in C^\infty(A^*)$ and $\alpha\in {\mathrm{HN}(\scG)}$, and $\mathcal{O}_\alpha$ a symplectic leaf through $\alpha$. There is a smooth path $\gamma\colon (-\epsilon, \epsilon)\to \mathcal{O}_\alpha\subset \mathrm{HN}(\scG)$ such that $\gamma(0)=\alpha$ and $\dot{\gamma(t)}=\mathcal{H}^G(\gamma(t))$ (here $\mathcal H^G=\{G\,, \cdot\}\in \mathfrak X(A^*)$). We have $0=F(\gamma(t))$ for all $t$ and \begin{align*}
    0=\dot{F(\gamma(t))}&=TF(\dot{\gamma(t)})\\&=TF(\mathcal H^G(\gamma(t)))\\&=\mathcal{H}^G[F](\gamma(t))=\{G, F\}(\gamma(t)).
\end{align*}Thus, $\{G, F\}(\alpha)=0$ and $\{G, F\}\in \mathcal I_{\mathrm{HN}(\scG)}$.
\end{proof}

\begin{remark}
    For a singular foliation $\rho(\Gamma(A))=\mathscr{G}$ on a smooth manifold $M$, we have \[\mathrm{HN}(\mathscr{G})\cap A^*|_{M_\mathrm{reg}}=X_{\mathcal{F}(\mathcal G)}\cap A^*|_{M_\mathrm{reg}}=\bigsqcup_{m\in M_{\mathrm{reg}}}\mathrm{Im}(\rho^*_m).\]
In general, the fibers $\mathrm{HN}(\mathscr{G})\subset A^*$ are not vector spaces but cones, in contrast to the fibers of $X_{\mathcal{F}(\mathscr{G})}\subset A^*$, which are vector spaces of possibly varying dimensions. Consequently, the inclusion $\mathrm{HN}(\mathscr{G})\subset X_{\mathcal{F}(\mathscr{G})}$ is strict in general. See the example below.
\end{remark}

\begin{example} \label{ex:comparison}We return to Examples \ref{ex0}, \ref{ex1},  \ref{ex2} and \ref{ex:vanishing_order2}. Let $\mathscr{G}\subset \mathfrak X(M)$ be a projective singular foliation (Examples \ref{ex0}, \ref{ex1}) so that there is a vector bundle $A$ with $\Gamma(A)\simeq \mathscr{G}$. The Helffer-Nourrigat cone $\mathrm{HN}(\mathscr{G})$ is equal to $A^*$, and it  coincides with the underlying space $X_{\mathcal{F}(\mathscr{G})}$ of the Poisson $C^\infty$-scheme associated with $\mathscr{G}$. If $\mathscr{G}$ is the singular foliation given by the action of $\mathfrak{so}(3)$ on $\mathbb{R}^3$ as in Example   \ref{ex2}, then \begin{align*}
    \mathrm{HN}(\mathscr{G})&=\bigr(\{0\}\times\{\ell^\perp, \ell \in \mathbb{P}^2\} \bigr)\bigsqcup_{(x, y, z)\neq 0}T_{\frac{(x, y, z)}{\lVert (x,y,z) \rVert}}^*S^2\\&=(\{0\}\times \mathbb R^3) \bigsqcup_{(x, y, z)\neq 0}T_{\frac{(x, y, z)}{\lVert (x,y,z) \rVert}}^*S^2=X_{\mathcal{F}(\mathscr{G})}\subset \mathbb{R}^3\times \mathbb{R}^3.
\end{align*}

However, if $\mathscr{G}$ is the singular foliation generated by vector fields on $\mathbb{R}^2$ that vanish to second order as in Example \ref{ex:vanishing_order2}, there is a strict inclusion $\mathrm{HN}(\mathscr{G})\subsetneq X_{\mathcal{F}(\mathscr{G})}$ because the fiber of $X_{\mathcal{F}(\mathscr{G})}$ over $(0,0)$ is the whole space $A^*_{(0,0)}\simeq \mathbb{R}^6$, but the fiber of $\mathrm{HN}(\mathscr{G})$ over $(0,0)$ is \[\{(a, b)\in \mathbb{R}^3\times \mathbb{R}^3\mid \exists \text{ a line}\; L\subset C\; \text{with}\; a, b\in L\}\subsetneq \mathbb{R}^6\]where $C$ is the cone in $\mathbb{R}^3$ given by $uw=v^2$, where $u, v$ and $w$ are coordinates on $\mathfrak {so}(3)^*\simeq \mathbb{R}^3$. See \cite{LLL1, Lo3}.
\end{example}

\appendix
\section{Some technical lemmas on $\mathrm{Spec}(\mathscr{A})$}\label{appen:technical-lem}

Recall that Dubuc's spectrum functor $\Spec$ assigns to a $\cin$-ring $\scA$ a topological space $X_\scA$ together with a sheaf $\scO_\scA$ of local $\cin$-rings.  That is,
\[
\Spec(\scA) = (X_\scA, \scO_\scA).
\]
The 
underlying topological space $X_\scA$ is the set of $\R$-points of $\scA$.  That is 
\[
X_\scA: =\{ p:\scA\to \R \mid p \textrm{ is a map of $\cin$-rings }\}.
\]
A basis for the topology on $X_\scA$ is the collection of sets $\{W_a\}_{a\in \scA}$ where 
\[
W_a:= \{ p\in X_\scA \mid p(a) \not= 0\}
\]
Moreover for any map $f:\scB\to \scA$ of $\cin$-rings the induced map $f^*: X_\scA \to X_\scB$ is continuous.   See \cite{Joy} for details.   
\begin{lemma} \label{lem:A1}
Let $M$ be a second countable Hausdorff manifold, $S\subseteq M$ a closed subset, $I_S$ is the corresponding vanishing ideal
\[
I(S):= \{f\in \cin(M) \mid f|_S =0\}
\]    
and $\scA = \cin(M)/I(S) \equiv \cin(M)|_S$ the corresponding quotient $\cin$-ring, ``the ring of smooth functions on $S$."   Then the space $X_\scA$ of $\R$-points of $\scA$ is homeomorphic to the set $S$ with its subspace topology:
\begin{equation}
    X_\scA \simeq Z_J.
\end{equation}
\end{lemma}
We'll see shortly that Lemma~\ref{lem:A1} is a corollary of Lemma~\ref{lem:A2} below.
\begin{lemma}\label{lem:A2}
Let $M$ be a second countable Hausdorff manifold.  For any ideal $J$ in the $\cin$-ring $\cin(M)$ of smooth functions on $M$ the space $X_{\cin(M)/J}$ of $\R$-points of the ring $\scA: =\cin(M)/J$ is homeomorphic to the closed set
\[
Z_J =\{x\in M \mid j(x) = 0 \textrm{ for all }j\in J\}
\]
with its subspace topology.
\end{lemma}

\begin{proof}[Proof that Lemma~\ref{lem:A2} implies Lemma~\ref{lem:A1}]  
We argue that 
\begin{equation} \label{eq:A.2.1}
    S = Z_{I(S)}.
\end{equation}
If $q\in S$ and $f\in I(S)$ then $f(q) =0$.  Hence $S\subseteq Z_{I(S)}$.   If $q\not\in S$ then, since $S$ is closed, there is a smooth function $h\in \cin(M)$ so that $h(q) =1$ and $h|_S$ is identically zero.  Thus $q\not \in Z_{I(S)}$. Consequently $M\smallsetminus S \subseteq M\smallsetminus Z_{I(S)}$.
We conclude that \eqref{eq:A.2.1} holds. 
\end{proof}

Before proceeding with our proof of Lemma~\ref{lem:A2} we recall a theorem of Pursell \cite{Pursell}, which is better known as ``Milnor's exercise" and one more result.
\begin{lemma} \label{lem:A3}
Let $M$ be a second countable Hausdorff manifold.  Then the map
\[
\ev: M \to X_{\cin(M)},\qquad x \mapsto \ev_x,
\]
where $\ev_x: \cin(M)\to \R$ is the evaluation map: $
\ev_x(f) = f(x)$
for all $f\in \cin(M)$, is a homeomorphism.
\end{lemma}
\begin{proof}
    See \cite[Theorem~2.1]{NGSS}.
\end{proof}
The following result is Lemma~5.5 from \cite{L-Poisson}.  We reproduce its proof for the reader's convenience.
\begin{lemma} \label{lem:A4}
Let $\varpi:\scB \to \scA$ be a {\em surjective } map of $\cin$-rings.
Then the induced map on spaces $
\varpi^*:X_\scA \to X_\scB$, $\varpi^*p:= p\circ \varpi$ is a closed embedding.
\end{lemma}
\begin{proof}

For an ideal $I\subset \scB$ define the set of its zeros to be
\[
Z_I:= \{y\in X_\scB \mid y (f) = 0 \textrm{ for all } f\in I\}.
\]
Since \[
Z_I = \bigcap _{f\in I} \{x\in \scB \mid x(f) =0\}, 
\]
the set $Z_I$ is closed in
the set of points $X_\scB$.
For any map of $\cin$-rings $\varphi:\scB\to
\scA$ the induced map $\varphi^*:X_\scA \to X_\scB$, $\varphi^*p  = p\circ \varphi$, is
continuous (this is part of the construction of the functor $\Spec$ on morphisms, see \cite{Joy}).

Let $J:= \ker \varpi$. By the first isomorphism theorem
\[
\varpi^*\big( X_\scA \big)= \{y\in X_\scB \mid y|_J = 0\},
\]  
which is $Z_J$.  Thus $\varpi^*$ maps $X_\scA$ continuously onto $Z_J$.
Since $\varpi$ is onto, $\varpi^*$ is injective. To prove that
$\varpi^*:X_\scA\to Z_J$ is a homeomorphism, it's enough to prove that
$\varpi^*$ is an open map.  A basic open set in $X_\scA$ is of the form
$\{a \not = 0\}: = \{ p\in X_\scA \mid p(a) \not = 0\}$.   Since $\varpi$
is onto, there is $b\in \scB$ such that $\varpi(b) = a$.     Then
\[
p (\varpi(b)) \not = 0 \quad \Leftrightarrow \quad (p\circ \varpi) (b) \not =0.
\]
Consequently
\[
\varpi^*\, (\{ a \not = 0\}) = \{ y \in X_\scB \mid  y = \varpi^* p \,
\textrm{ and } \,  p(a) \not = 0\} = Z_J \cap \{ b \not =0\}, 
\]
and we are done.
\end{proof}

\begin{proof}[Proof of Lemma~\ref{lem:A2}] 
Lemma~\ref{lem:A4} applied to $\scB= \cin(M)$ and the quotient map $\varpi: \scB = \cin(M)\to \cin(M)/J = \scA$ implies that 
\[
\varpi^*:X_\scA \to X_{\cin(M)}
\]
is a closed embedding.  Since $\ev\inv: X_{\cin(M)}\to M$ is a homeomorphism the composite map
\[
\Psi: = \ev\inv \circ \varpi^*: X_\scA \to M
\]
is a closed embedding.    It remains to check that $\Psi(X_\scA) = Z_J$.  If $p\in X_\scA$ then 
\[
(\varpi^*p)(j) = p(\varpi(j)) = p(0) = 0
\]
for all $j\in J$.  Hence if $z\in \ev\inv (\varpi^*(X_\scA))$ then $j(z) = \ev_z(j) = 0$ for all $j\in J$.  Thus $\Psi(X_\scA) \subseteq Z_J$.  

Conversely if $z\in Z_J$ then $\ev_z(j) = 0$ for all $j\in J$. Hence the map $\ev_z:\cin(M)\to \R$ descends to a unique map $p:\cin(M)/J \to \R$ with $p\circ \varpi = \ev_z$.  Hence $\Psi(X_\scA) = Z_J$.
\end{proof}

\section{Nagano-Sussmann orbit theorem and zero locus}


The following theorem is crucial for discussing leaves in singular spaces.
\begin{theorem}[Nagano-Sussmann \cite{Nag}]\label{thm:Nagano–Sussman} Let $\mathcal{V}\subseteq \mathfrak X(N)$ be a locally finitely generated $C^\infty(N)$-submodule of vector fields on a manifold $N$. For every $\ell\in N$, the set\footnote{We implicitly assume that the flows are defined.}
       $$\mathcal{O}_\ell:=\left\{\phi_{t_1}^{Z_1}\circ \phi_{t_2}^{Z_2}\circ \cdots \circ \phi_{t_n}^{Z_n}(\ell),\; t_1, \ldots, t_n\in \mathbb R,\; Z_1,\ldots, Z_n\in \mathcal{V},\;n\in \mathbb N\right\}$$
is a connected immersed submanifold of $N$. The sets $\mathcal{O}_\ell, \;\ell\in N$, form a partition of $N$.
\end{theorem}

We conclude the section with two technical lemmas
\begin{lemma}\label{lem:flow-zero-locus}Let $M$ be a smooth manifold and $\mathcal{I}\subset C^\infty(M)$ be a locally finitely generated ideal. Let $X\in\mathfrak X(M)$ such that $X[\mathcal{I}]\subset \mathcal{I}$. Then any integral curve $\gamma \colon J \to M$ of $X$ starting at $p\in Z_{\mathcal{I}}$ is in $Z_\mathcal{I}$ for all $t\in J$.
\end{lemma}

\begin{proof}
Let $\gamma\colon J=(-\epsilon,\epsilon) \to M$ be an integral curve of $X$ with $\gamma(0)\in Z_{\mathcal{I}}$, and let $J'$
be the connected component containing $0$ of the closed subset
\(
\gamma^{-1}(Z_{\mathcal I})\subset J.
\)
We show that $J'=J$.

$J'$ is non empty (it contains $0$) and is closed in $J'$ by definition. Let us check that $J'$ is open in $J$. Let $t_0\in J'$ and
set $q=\gamma(t_0)\in Z_{\mathcal I}\). Since $\mathcal I$ is locally finitely generated, there exists an open neighborhood \(\mathcal U\) of \(q\)
such that
\[
\mathcal I|_{\mathcal U}=(f_1,\ldots,f_r)
\]
for some \(f_1,\ldots,f_r\in C^\infty(\mathcal U)\).
Moreover, the assumption \(X[\mathcal I]\subset\mathcal I\) implies
that there exist smooth
functions \(c_{ij}\in C^\infty(\mathcal U)\) such that
\[
Xf_i=\sum_{j=1}^r c_{ij}f_j,
\qquad i=1,\ldots,r.
\] 
By continuity of \(\gamma\), there is an open interval
\(K\subset J\), containing \(t_0\), such that
\(\gamma(K)\subset\mathcal U\). Consider the vector-valued function \[t\in K\mapsto(f_1(\gamma(t)),\ldots, f_r(\gamma(t))).\]
For every $i\in\{1,\ldots, r\}$, \[
\frac{d}{dt}f_i(\gamma(t))
=Tf_i(X(\gamma(t)))=(Xf_i)(\gamma(t))
=\sum_{j=1}^r c_{ij}(\gamma(t))\,f_j(\gamma(t)).
\]
Hence, $V(t)=(f_1(\gamma(t)),\ldots, f_r(\gamma(t)))$ is the solution of the linear ordinary differential equations \begin{equation}\label{eq:ODE}
F'(t)=\mathfrak M(t)F(t)\end{equation} where $\mathfrak M(t)=(c_{ij}(\gamma(t)))$ is a $r\times r$ matrix. Since $q=\gamma(t_0)\in Z_{\mathcal I}$, we have $V(t_0)=0=V'(t_0)$. By uniqueness, the
solution of \eqref{eq:ODE} with $F(t_0)=0=F'(t_0)$ is identically zero. Hence we have $V(t)=(f_1(\gamma(t)),\ldots, f_r(\gamma(t)))=0$ for all $t\in K$. Since $f|_\mathcal{U}$ is a $C^\infty(\mathcal U)$-linear combination of the $f_i$'s, it implies that $f(\gamma(t))=0$. Therefore, $\gamma(t)\in Z_{\mathcal{I}}$ for all $t\in K$. It follows that $K\subset \gamma^{-1}(Z_\mathcal{I})$. Since \(K\) is connected and contains \(t_0\in J'\), it is contained in the connected component \(J'\). Thus \(J'\) is open in \(J\).  Therefore, \(J'\) is nonempty, open, and closed in \(J\). Since \(J\) is an interval, hence connected, we conclude that $J'=J$. Consequently, \[ \gamma(t)\in Z_{\mathcal I} \] for all \(t\in J\), as desired.

\end{proof}

\begin{remark}
    It is not true in general, for $X$ a smooth vector field and $\mathcal I$ an ideal of smooth functions that  $X[\mathcal I] \subset \mathcal I$ implies $(\Phi_t^X)^* [\mathcal I] \subset \mathcal I $. For instance, take $ M= \mathbb R$, $\mathcal I $ to be the ideal of functions vanishing on $\mathbb R_- $ and $X={\partial_x }$. Here $Z(\mathcal{I})=(\infty, 0]$ and for $t>0$,  $\Phi^{{\partial_x}}_t((\infty, 0])\nsubseteq (\infty, 0]$.  It is therefore crucial, in the previous argument, to make use of the fact that $\mathcal I$ is locally finitely generated.
\end{remark}

\begin{lemma}\label{lem:Pushforward-of-hamiltonian}
    Let $M$ be a smooth manifold, and let $\pi = \{\cdot\,, \cdot\}$ be a biderivation on $C^\infty(M)$. Suppose there exists a locally finitely generated ideal $\mathcal I \subseteq C^\infty(M)$ such that 

\begin{enumerate}
    \item\label{item:(A.4.i)} $\{C^\infty(M), \mathcal{I}\}\subset \mathcal{I}$,
    \item the Jacobiator $J_\pi \colon C^\infty(M)^{\otimes 3} \to C^\infty(M)$, defined by
\[
J_\pi(F, G, H) := \{\{F, G\}, H\} + \{\{G, H\}, F\} + \{\{H, F\}, G\},
\]
is $\mathcal I$-valued; that is, $J_\pi(F, G, H) \in \mathcal I$ for all $F, G, H \in C^\infty(M)$.
\end{enumerate}For any $F \in C^\infty(M)$, denote by $X^F := \{F, \cdot\} \in \mathfrak{X}(M)$ the associated vector field. 

Then, for all $F, G \in C^\infty(M)$ and all $t \in \mathbb{R}$ for which the flow $\Phi_t^{X^F}$ is defined, the (local) vector field
\[
(\Phi_t^{X^F})_* (X^G)- X^{G \circ \Phi_{t}^{X^F}}
\]
vanishes on the zero locus $Z(\mathcal{I})$.
\end{lemma}

\begin{proof} To start with, we show the following identity. For $F, G, H\in C^\infty(M)$ \[\bigr\{G\circ\Phi_{t}^{X^F},\;H\circ \Phi_{t}^{X^F}\bigr\}|_{Z(\mathcal{I})}=\{F, G\}\circ \Phi_{t}^{X^F}|_{Z(\mathcal{I})}.\] Indeed, by differentiating with respect to $t$ we get:
\begin{align*}
&\frac{d}{dt}
\bigr(\Phi_{-t}^{X^F}\bigr)^{*}\!\bigr\{\bigr(\Phi_{t}^{X^F}\bigr)^{*}G,\;\bigr(\Phi_{t}^{X^F}\bigr)^{*}H\bigr\}=\\
&\bigr(\Phi_{-t}^{X^F}\bigr)^{*}\underbrace{\bigr(
\bigr\{(X^F\bigr(\Phi_{t}^{X^F}\bigr)^{*}G,\bigr(\Phi_{t}^{X^F}\bigr)^{*}H\bigr\}
+
\bigr\{\bigr(\Phi_{t}^{X^F}\bigr)^{*}G,X^F\bigr(\Phi_{t}^{X^F}\bigr)^{*}H)\bigr\}
-
X^{F}\bigr\{\bigr(\Phi_{t}^{X^F}\bigr)^{*}G,\bigr(\Phi_{t}^{X^F}\bigr)^{*}H\bigr\}
\bigr)}_{J_\pi\bigr(F, \bigr(\Phi_{t}^{X^F}\bigr)^{*}G, \bigr(\Phi_{t}^{X^F}\bigr)^{*}H\bigr)\in \mathcal{I}}.
\end{align*}
We have used the fact that the failure of $X^F$ to be a derivation of $\{\cdot\,,\cdot\}$ is valued in $\mathcal I$. 
By Lemma \ref{lem:flow-zero-locus}, item \ref{item:(A.4.i)} implies that $\Phi_{-t}^{X^F}(x_0)\in Z(\mathcal{I})$ for every $x_0\in Z(\mathcal{I})$ where $\Phi^{X^F}_{-t}(x_0)$ is defined. Therefore, \[\frac{d}{dt}
\bigr(\Phi_{-t}^{X^F}\bigr)^{*}\!\bigr\{\bigr(\Phi_{t}^{X^F}\bigr)^{*}G,\;\bigr(\Phi_{t}^{X^F}\bigr)^{*}H\bigr\}(x_0)=0.\] Hence, the map $t\mapsto \bigr\{\bigr(\Phi_{t}^{X^F}\bigr)^{*}G,\;\bigr(\Phi_{t}^{X^F}\bigr)^{*}H\bigr\}\circ \Phi_{-t}^{X^F}(x_0) $ is constant and is equal to the value at $t=0$ \[\bigr\{\bigr(\Phi_{t}^{X^F}\bigr)^{*}G,\;\bigr(\Phi_{t}^{X^F}\bigr)^{*}H\bigr\}\circ \Phi_{-t}^{X^F}(x_0)=\{F, G\}(x_0).\] This implies that \[\bigr\{G\circ\Phi_{t}^{X^F},\;H\circ \Phi_{t}^{X^F}\bigr\}|_{Z(\mathcal{I})}=\{F, G\}\circ \Phi_{t}^{X^F}|_{Z(\mathcal{I})}.\]This proves the identity.\\

For $x_0\in Z(\mathcal{I})$,  \begin{align*}
    \bigr(\Phi_{t}^{X^F}\bigr)_{*}(X^G)[H](x_0)&=T\Phi^{X_F}_{-t}(X^G\circ \Phi^{X^G}_t)[H](x_0)\\&=T(H\circ \Phi_{-t}^{X^F})(X^G\circ \Phi_{t}^{X^F}(x_0))\\&=X^G[H\circ \Phi_{-t}^{X^F}]\circ \Phi_{t}^{X^F}(x_0)=\{G, H\circ \Phi_{-t}^{X^F}\}\circ \Phi_{t}^{X^F}(x_0)\\&=\{G\circ\Phi_{t}^{X^F}, H\}(x_0)\\&=X^{G\circ\Phi_{t}^{X^F}}[H](x_0).
\end{align*}

\end{proof}
\end{document}